\documentclass[10pt,a4paper,openany,oneside]{article}
\usepackage[a4paper,left=3cm,right=3cm, top=3cm, bottom=3cm]{geometry}
\usepackage{times}
\usepackage{authblk}
\usepackage[latin1]{inputenc}
\usepackage[english]{babel}
\usepackage{mathrsfs}
\usepackage{stmaryrd}
\usepackage{amssymb,amsmath,amsthm}
\usepackage{hyperref}
\usepackage{graphicx}
\usepackage{subcaption}
\usepackage{eurosym}
\usepackage{color}
\usepackage{enumerate}
\usepackage{multirow}
\usepackage{stmaryrd}
\usepackage{booktabs}
\usepackage{float}
\theoremstyle{plain}
\newtheorem{thm}{Theorem}[section]

\newtheorem{prop}[thm]{Proposition}

\theoremstyle{definition}
\newtheorem{defn}{Definition}[section]
\theoremstyle{remark}
\newtheorem{remark}{Remark}

\newcommand{\C}{\mathbb C}
\newcommand{\R}{\mathbb R}
\newcommand{\N}{\mathbb N}
\newcommand{\Oo}{\mathcal O}

\newcommand{\PT}{\mathcal P\mathcal T}

\def\eps{\varepsilon}

\def\ri{{\rm i}}

\def\bi{\begin{itemize}}
\def\ei{\end{itemize}}

\newcommand{\B}{\mathcal{B}}

\newcommand{\Z}{\mathbb{Z}}

\newcommand{\cD}{{\mathcal D}}
\newcommand{\cE}{{\mathcal E}}

\newcommand{\cH}{{\mathcal H}}

\newcommand{\pa}{\partial}

 \def\dd{\, {\rm d}}

\DeclareMathOperator{\spann}{span}

\DeclareMathOperator*{\esssup}{ess\,sup}

\DeclareMathOperator{\Real}{Re}
\DeclareMathOperator{\Imag}{Im}

\DeclareOldFontCommand{\it}{\normalfont\itshape}{\mathit}

\newcommand{\bspm}{\left(\begin{smallmatrix}}\newcommand{\espm}{\end{smallmatrix}\right)}
\newcommand{\bpm}{\begin{pmatrix}}\newcommand{\epm}{\end{pmatrix}}
\newcommand{\trans}{{\mathsf T}}

\def\blem{\begin{lemma}}\def\elem{\end{lemma}}
\def\bthm{\begin{theorem}}\def\ethm{\end{theorem}}
\def\bcor{\begin{corollary}}\def\ecor{\end{corollary}}
\def\beq{\begin{equation}}\def\eeq{\end{equation}}
\def\beqq{\begin{equation*}}\def\eeqq{\end{equation*}}
\def\bal{\begin{align}}\def\eal{\end{align}}
\def\bpf{\begin{proof}}\def\epf{\end{proof}}
\def\bex{\begin{example}}\def\eex{\end{example}}
\def\brem{\begin{remark}}\def\erem{\end{remark}}
\def\bass{\begin{assumption}}\def\eass{\end{assumption}}
\def\bprop{\begin{proposition}}\def\eprop{\end{proposition}}
\def\bdefi{\begin{definition}}\def\edefi{\end{definition}}

\begin{document}

\title{Eigenvalue Bifurcation in Doubly Nonlinear Problems with an Application to Surface Plasmon Polaritons}
\date{\today}
\author{Tom\'a\v{s} Dohnal$^*$ and Giulio Romani}
\affil{Institut f\"{u}r Mathematik,  Martin-Luther-Universit\"{a}t Halle-Wittenberg,\\ 06099 Halle (Saale), Germany\\
\small{tomas.dohnal@mathematik.uni-halle.de$^*$, giulio.romani@mathematik.uni-halle.de}}

\maketitle
\abstract{We consider a class of generally non-self-adjoint eigenvalue problems which are nonlinear in the solution as well as in the eigenvalue parameter (``doubly'' nonlinear). We prove a bifurcation result from simple isolated eigenvalues of the linear problem using a Lyapunov-Schmidt reduction and provide an expansion of both the nonlinear eigenvalue and the solution. We further prove that if the linear eigenvalue is real and the nonlinear problem $\PT$-symmetric, then the bifurcating nonlinear eigenvalue remains real.
These general results are then applied in the context of surface plasmon polaritons (SPPs), i.e. localized solutions for the nonlinear Maxwell's equations in the presence of one or more interfaces between dielectric and metal layers. We obtain the existence of transverse electric SPPs in certain $\PT$-symmetric configurations.}
\vskip0.2truecm

\section{Introduction}\label{S:intro}

We study the nonlinear problem
\begin{equation}\label{eq}
L(x,\omega)\varphi:=A\varphi-W(x,\omega)\varphi=f(x,\omega,\varphi),\quad x\in\R^d,
\end{equation}
where $A:L^2(\R^d,\C)\supset D(A)\to L^2(\R^d,\C)$ is a densely defined, closed (possibly non-self-adjoint) operator with a non-empty resolvent set. Throughout the paper the space $L^2(\R^d,\C)$ and all other function spaces are complex vector spaces, i.e. defined over the complex field $\C$. The potential $W$ is generally nonlinear in the spectral parameter $\omega$ and typically complex valued. The function $f$ is nonlinear in both $\omega$ and $\varphi$ and is asymptotically equivalent to a monomial near $\varphi=0$. Moreover, we suppose that $f$ is Lipschitz continuous in a neighbourhood of an eigen-pair $(\omega_0,\varphi_0)$, where $\omega_0 \in \C$ is a simple isolated eigenvalue of $L(x,\cdot)$. We prove the bifurcation from $\omega_0$ using a fixed point argument and a Lyapunov-Schmidt decomposition. Bifurcation from simple eigenvalues is a well studied problem even in the non-selfadjoint case \cite{CR,Dancer,Ize,Niren}. In particular, bifurcation in complex Banach spaces (as relevant in our problem) is investigated in \cite{Dancer,Ize} by means of a Lyapunov-Schmidt reduction coupled with topological degree techniques. However, our result includes also an asymptotic expansion of $(\omega,\varphi)$ depending on the behaviour of $f$ for small $\varphi$ and for $\omega$ near $\omega_0$. 
More precisely, we find a solution $(\omega,\varphi)\in\C\times D(A)$ of the form
\begin{equation*}
\omega=\omega_0+\varepsilon\nu+\varepsilon^{1+\tau}\sigma,\qquad\quad\varphi=\varepsilon^\alpha\varphi_0+\varepsilon^{\alpha+1}\phi+\varepsilon^{\alpha+1+\tau}\psi,
\end{equation*}
where $\varepsilon>0$ is small, $\omega$ is the spectral parameter, $\alpha$ is related to the degree of homogeneity of $f(x,\omega,\cdot)$ near $0$, $\tau$ is a positive parameter, and $\nu,\sigma\in\C$ as well as $\phi,\psi\in D(A)\cap\langle\varphi_0^*\rangle^\perp$ are uniquely determined. Moreover, $\nu$ is explicit, see \eqref{nu_GEN}, and $\phi$ satisfies the linear equation \eqref{phi_GEN}.

In \cite{DS} the bifurcation was proved (and an asymptotic expansion of $(\omega,\varphi)$ was provided) for 
\begin{equation}\label{eq_DS}
(A-\omega)\varphi=\eps f(\varphi)
\end{equation}
with $A$ as above. This problem clearly has a linear dependence on the spectral parameter $\omega$. The coefficient $\eps$ in \eqref{eq_DS} is the bifurcation parameter and one studies the bifurcation from an eigenvalue $\omega_0$ at $\eps=0$. As the bifurcation parameter appears explicitly in the equation, the form of the asymptotic expansion of $(\omega,\varphi)$ is unique. Note that \eqref{eq_DS} can be rescaled to $(A-\omega)\psi=f(\psi)$ only for the case of homogeneous nonlinearities $f$. Therefore, our result extends that of \cite{DS} to the case of more general nonlinearities $f$ and a nonlinear dependence of both $f$ and $W$ on the spectral parameter.

An important application of non-selfadjoint problems which are  nonlinear in $\omega$ is the propagation of electromagnetic waves in dispersive media, in particular in structures that include a metal. Interfaces of two different media can support localized waves. A typical example is a surface plasmon polariton (SPP) at the interface of a dielectric and a metal, see e.g. \cite{Rae,Pitarke_2006} or, when more layers of dielectrics and/or metals are considered, \cite{YXH,WRY,Han_2014}. The general case is, of course, described by Maxwell's equations. Assuming the absence of free charges, we have
\begin{equation}\label{Max}
\begin{split}
\mu_0\partial_t \mathcal H=-\nabla\times\mathcal E,\qquad\epsilon_0\partial_t \mathcal D=\nabla\times \mathcal H,\qquad\nabla\cdot \mathcal D=\nabla\cdot \mathcal H=0,
\end{split}
\end{equation}
where $\cE$ and $\cH$ is the electric and magnetic field respectively, $\cD=\cD(\cE)$ is the electric displacement field and $\epsilon_0$ and $\mu_0$ are respectively the permittivity and the permeability of the free space.
The displacement field $\mathcal D$ is generally nonlinear in $\mathcal E$ and non-local in time. For odd (e.g. Kerr) nonlinearities and a monochromatic field $(\mathcal E,\mathcal H,\mathcal D)(x,y,z,t)=(E,H,D)(x,y,z)e^{\ri\omega t}+ \text{c.c.}$ (with a real frequency $\omega$) a nonlinear eigenvalue problem in $(\omega, (E,H))$ is obtained if higher harmonics are neglected, for details see Sec. \ref{Section_SPPs}. Equation \eqref{Max} as well as the eigenvalue problem have to be accompanied by interface conditions if an interface of two media is present. Assuming that the interface is planar and parallel to the $yz$-plane, the interface conditions are
\beq\label{E:IFC}
\llbracket E_2\rrbracket =\llbracket E_3\rrbracket =\llbracket D_1\rrbracket =0,\, \llbracket H\rrbracket =0,
\eeq
where we define (for the interface located at $x=x_0$), $\llbracket E_2\rrbracket :=E_2(x\to x_0^+)-E_2(x\to x_0^-),$ etc., see Sec. \ref{Section_SPPs}. 

A simple example in the cubically nonlinear case is obtained in structures independent of the $y,z$-variables by choosing the transverse electric (TE) ansatz 
\begin{equation}\label{E:ans-SPP}
E(x,y,z)=(0,0,\varphi(x))^T e^{\ri ky}
\end{equation}
with $k\in\R$. It leads to the scalar nonlinear problem
\begin{equation}\label{Schrod}
\varphi'' + W(x,\omega)\varphi +\Gamma(x,\omega)|\varphi|^2\varphi=0, \quad x\in \R
\end{equation}
with functions $W,\Gamma:\R^2\to\C$. The interface conditions here boil down to the continuity condition on $\varphi$ and $\varphi'$, see Section \ref{Section_SPPs}.

The bifurcation result provides a curve $\eps \mapsto \big(\omega(\eps),\varphi(\eps)\big)$ with $\omega(0)=\omega_0$ and $\varphi(0)=0$. Even if $\omega_0$ is real, the curve can lie in $\C\setminus \R$ for all $\eps\neq 0$. In order for $\omega$ to correspond to the (real) frequency of an electromagnetic field, one needs to ensure that the curve lies in $\R$. As we show in Sec. \ref{S:PT}, this is possible by restricting the fixed point argument to a symmetric subspace, namely the $\PT$-symmetric subspace. $\PT$-symmetry has been studied extensively in quantum mechanics, see e.g. \cite{Ben,CY2018}. Recently, a number of physics papers have studied nonlinear $\PT$-symmetric problems from a phenomenological point of view mainly with emphasis on localized solutions, e.g. \cite{Rubinstein-2007-99,Zezyulin-2012-85,CY2018}. In the context of SPPs, where metals normally lead to a lossy propagation, $\PT$-symmetry has been applied to obtain lossless propagation, see \cite{AD14,Barton2018}. Mathematically, the restriction of a fixed point argument to a $\PT$-symmetric (or more generally antilinearly symmetric) subspace has been used to obtain real nonlinear eigenvalues, see, e.g., \cite{Rubinstein-2010-195,DS,DP}.

This article is organized as follows. In Section \ref{SectionBIFURC} we state and prove our main bifurcation result (Theorem \ref{Theorem_BIF_NLeigv}). The realness of the nonlinear eigenvalue is ensured in the case of $\PT$-symmetry in Section \ref{S:PT}. Applications to SPPs are then given in Section \ref{Section_SPPs}, where 2- and 3-layer-configurations are investigated.

\section{Bifurcation of nonlinear eigenvalues}\label{SectionBIFURC}

In this section we study problem \eqref{eq}, where $A$, $W$ and $f$ satisfy assumptions (A1)-(f4) below. We prove the existence of a branch of solutions starting from an eigenpair $(\omega_0,\varphi_0)\in\C\times D(A)$, i.e. $L(\cdot,\omega_0)\varphi_0=0$, such that
\begin{enumerate}[\text{E}1)]
	\item $\omega_0$ is \textit{algebraically simple} in the sense that $\kappa=0$ is an algebraically simple eigenvalue of the standard eigenvalue problem $L(\cdot,\omega_0)u=\kappa u$, i.e. $\ker(L(\cdot,\omega_0)^2)=\ker(L(\cdot,\omega_0))=\langle\varphi_0\rangle$,
	\item $\omega_0$ is \textit{isolated} in the sense that $\kappa=0$ is an isolated eigenvalue of the problem $L(\cdot,\omega_0)u=\kappa u$. 
\end{enumerate}

\underline{Notation}: Henceforth, the norm and the inner product in the underlying Hilbert space $L^2(\R^d,\C)$ will be denoted by $\|\cdot\|$ and $\langle \cdot,\cdot\rangle$ respectively. Moreover, $\|\cdot\|_\infty$ stands for the $L^\infty(\R^d)$ norm and $\|\cdot\|_A$ for the graph norm, i.e. $\|u\|_A:=\|u\|+\|Au\|$. Note that $D(A)=D(L)$ due to assumptions (A2) and (W1).
\vskip0.2truecm

We define $\varphi_0^*$ as the\footnote{Recall that if $0$ is a simple isolated eigenvalue of $L(\cdot,\omega_0)$, then it is also a simple isolated eigenvalue of $L(\cdot,\omega_0)^*$, cf. \cite[Chap.III.6.5-6]{Kato}} eigenfunction of $\big(L(\cdot,\omega_0)\big)^*$ and we suppose that the normalizations of the eigenfunctions $\varphi_0$ and $\varphi_0^*$ are chosen such that
\begin{equation*}
\|\varphi_0\|=1,\qquad\langle\varphi_0,\varphi_0^*\rangle=1.
\end{equation*}
The latter normalization is allowed since the simplicity of $\omega_0$ ensures that $\langle\varphi_0,\varphi_0^*\rangle\neq 0$. Indeed, if $\langle \varphi_0,\varphi_0^* \rangle =0$, then $\varphi_0 \in {\rm Ker}(L(\cdot,\omega_0)) \cap {\rm Ran}(L(\cdot,\omega_0))$ since ${\rm Ker}(L(\cdot,\omega_0)^*)^\perp = {\rm Ran}(L(\cdot,\omega_0))$. From $L(\cdot,\omega_0)\psi=\varphi_0$ for some $\psi\in D(A)$ we get $L(\cdot,\omega_0)^2\psi=0$ and the algebraic simplicity in (E1) implies $\psi = c\varphi_0$ and hence $\varphi_0=0$, which is a contradiction.

\medskip
\noindent We consider the following assumptions on the operator $A$ and the potential $W$: there exists $\delta>0$ such that
\begin{enumerate}[\text{A}1)]
	\item $A: D(A)\to L^2(\R^d,\C)$ is a densely defined, closed operator with a non-empty resolvent set;
	\item $D(A)\hookrightarrow L^\infty(\R^d,\C)$, where the embedding is continuous;
\end{enumerate}
\begin{enumerate}[\text{W}1)]
	\item $W:\R^d\times \C\to \C$ satisfies that $W(x,\cdot)$ is holomorphic on $B_\delta(\omega_0)\subset \C$ for a.e. $x\in \R^d$ and there exists $c>0$ such that
	$$\|W(\cdot,\omega)\|_\infty\leq c \quad \forall \omega\in B_\delta(\omega_0);$$
\end{enumerate}
and the technical assumption
\begin{enumerate}[\text{Wt})]
	\item $\langle\partial_\omega W(\cdot,\omega_0)\varphi_0,\varphi_0^*\rangle\not=0$.
\end{enumerate}
Regarding the nonlinearity $f=f(x,\omega,\varphi)$, we assume that there are $\delta,\alpha,\lambda_0>0$ such that
\begin{enumerate}[\text{f}1)]
\item $f(\cdot,\omega,\varphi(\cdot))\in L^2(\R^d,\C)$ for all $\omega \in  B_\delta(\omega_0)$ and $\varphi \in B_\delta(0)\subset D(A)$,
	\item there exists a constant $K_f^\omega(\varphi_0)>0$ such that for any $\lambda\in(0,\lambda_0)$ there holds
	$$\|f(\cdot,\omega_1,\lambda\varphi)-f(\cdot,\omega_2,\lambda\varphi)\|\leq K_f^\omega(\varphi_0)\lambda^{\frac1\alpha+1}|\omega_1-\omega_2|,\quad\forall\omega_1,\omega_2\in B_\delta(\omega_0)\;,\forall\varphi\in B_\delta(\varphi_0);$$
	\item there exists a constant $K_f^\varphi(\omega_0)>0$ such that for any $\lambda\in(0,\lambda_0)$ there holds
	$$\|f(\cdot,\omega,\lambda\varphi_1)-f(\cdot,\omega,\lambda\varphi_2)\|\leq K_f^\varphi(\omega_0)\lambda^{\frac1\alpha+1}\|\varphi_1-\varphi_2\|,\quad\forall\varphi_1,\varphi_2\in B_\delta(\varphi_0)\;,\forall\omega\in B_\delta(\omega_0);$$
	\item for any $\omega\in B_\delta(\omega_0)$ there exists $g_\omega\in L^\infty(\R^d,\C)\setminus\{0\}$ such that  
	\begin{equation}\label{asympt_beh2}
	\left\|f(\cdot,\omega,\varphi)-g_\omega(\cdot)|\varphi|^\frac1\alpha \varphi\right\|=\Oo\left(\left\||\varphi|^{\frac1\alpha+1+\beta}\right\|\right)\quad\mbox{as}\quad \C \ni \varphi\to 0
	\end{equation}
	for some $\beta>0$, uniformly wrt $\omega\in B_\delta(\omega_0)$. Moreover, we assume that $g_\omega$ is Lipschitz in $\omega$ for $\omega\in B_\delta(\omega_0)$ uniformly wrt $x\in\R^d$, i.e. there exists a constant $K_g>0$ such that $$\|g_{\omega_1}-g_{\omega_2}\|_\infty\leq K_g|\omega_1-\omega_2| \quad\forall\omega_1,\omega_2\in B_\delta(\omega_0).$$
\end{enumerate}

\begin{thm}\label{Theorem_BIF_NLeigv}
	Suppose that (E1), (E2) hold, i.e. $\omega_0$ is an algebraically simple and isolated eigenvalue of $L$ with eigenfunction $\varphi_0$ and that $A$, $W$ and $f$ satisfy assumptions (A1)-(f4). Let also $\tau\in(0,\min\{1,\alpha\beta\}]$. Then there is a unique branch bifurcating from $(\omega_0,0)$. There exists $\varepsilon_0>0$ s.t. for any $\varepsilon\in(0,\varepsilon_0)$ the solution $(\omega,\varphi)$ normalized to satisfy $\langle\varphi,\varphi_0^*\rangle=\varepsilon^\alpha$ has the form
	\begin{equation}\label{omega_varphi}
	\omega=\omega_0+\varepsilon\nu+\varepsilon^{1+\tau}\sigma,\qquad\quad\varphi=\varepsilon^\alpha\varphi_0+\varepsilon^{\alpha+1}\phi+\varepsilon^{\alpha+1+\tau}\psi,
	\end{equation}
	with $\nu,\sigma\in\C$ and $\phi,\psi\in D(A)\cap\langle\varphi_0^*\rangle^\perp$.
\end{thm}

Before going into the details of the proof, let us give some remarks and examples of applications.

\begin{remark}
	The proof of Theorem \eqref{Theorem_BIF_NLeigv} is constructive. Therefore, we also know which problems the constants $\nu,\sigma$ and the functions $\phi,\psi$ satisfy. In particular, $\nu$ and $\phi$ are uniquely determined by \eqref{nu_GEN}, \eqref{phi_GEN} and $(\sigma,\psi)$ uniquely solves the nonlinear system \eqref{3_sys}, \eqref{4_sys} in $B_{r_1}(0)\times B_{r_2}(0)\subset\C\times D(A)$, for suitable $r_1,r_2=O(1)$ as $\varepsilon\to0$.
\end{remark}

\begin{remark}
	When $D(A)$ is the Sobolev space $H^s(\R^d,\C)$, assumption (A2) is equivalent to the requirement $s>d/2$.
\end{remark}

\begin{remark}\label{R:Taylor}
Note that assumption (W1) ensures that
\beq\label{W-der-est}
\|\pa_\omega^kW(\cdot,\omega_0)\|_{\infty} \leq c<\infty \quad \forall k  \in \N.
\eeq
Indeed, by Taylor's theorem for holomorphic functions (see Chapt. 4, Sec. 3.1 in \cite{ahlfors1953}) we have
\beq\label{E:Taylor-exp-W}
W(\cdot,\omega)=\sum_{j=0}^{n-1}\pa_\omega^jW(\cdot,\omega_0)(\omega-\omega_0)^j + T_n(\cdot,\omega)(\omega-\omega_0)^n,
\eeq
where 
$$T_n(\cdot,\omega)=\frac{1}{2\pi\ri} \int_{\pa B_r(\omega_0)}\frac{W(\cdot,z)}{(z-\omega_0)^n(z-\omega)}dz$$
for all $\omega \in B_r(\omega_0)$ and any $0<r<\delta$.
One can easily estimate
\begin{align}
\|T_n(\cdot,\omega)\|_\infty 
&\leq  \frac{\esssup\limits_{x\in \R^d}\max\limits_{z\in \pa B_r(\omega_0)}|W(x,z)|}{r^{n-1}(r-|\omega-\omega_0|)} \;\quad \text{for all}\;\; \omega\in B_{r'}(\omega_0) \;\ \text{if} \;\ 0<r'<r \nonumber\\
&\leq  \frac{2M}{r^n} \quad \text{for all}\quad \omega\in B_{r/2}(\omega_0), \label{E:Tn-est}
\end{align}
where
$$M:=\sup\limits_{\omega\in B_\delta(\omega_0)}\|W(\cdot,\omega)\|_\infty.$$ 
Note that $\omega\in B_{r/2}(\omega_0)$ is satisfied if $\eps>0$ is small enough. To estimate $\|\pa_\omega^kW(\cdot,\omega_0)\|_{\infty}, k \in\N$, one proceeds by induction using \eqref{E:Taylor-exp-W}, assumption (W1), and \eqref{E:Tn-est}.
\end{remark}

\begin{remark}
Assumption (Wt) is the classical transversality condition for the bifurcation from a simple eigenvalue \cite[Theorem 1]{CR}, \cite[Theorem 28.6]{deimlingNLFA} in the case 
\beq \label{E:fC2}f(x,\cdot,\cdot)\in C^1(\C\times \Omega),  \ \pa_\omega\pa_\varphi f(x,\cdot,\cdot)\in C(\C\times \Omega) \ \text{ for some neighborhood } \  \Omega\subset D(A) \ \text{ of zero.}
\eeq
However, note that in our setting the nonlinearity $f(\cdot,\omega,\varphi)\sim g_\omega(\cdot)|\varphi|^\frac1\alpha \varphi$ for $\varphi\to 0$ (and with $\varphi \in D(A)$ with a suitable $D(A)$, e.g. $D(A)=H^2(\R,\C)$ for $d=1$) is differentiable only at zero. This is due to the fact that our function spaces are defined over the complex field.

Nevertheless, given the differentiability property, if \eqref{E:fC2} holds, (Wt) is equivalent to $F_{\omega \varphi}(\omega_0,0)[\varphi_0]\notin \text{Ran}(F_\varphi(\omega_0,0))$, where $F(\omega,\varphi):=L(\cdot,\omega)\varphi-f(\cdot,\omega,\varphi)$. To see this, first note that
$$F_\varphi(\omega_0,0)=A-W(\cdot,\omega_0)$$
since $\pa_\varphi f(\cdot,\omega,0)=0$ for all $\omega\in B_\delta(\omega_0)$ by (f4), and
$$F_{\omega \varphi}(\omega_0,0)[\varphi_0]=-\pa_\omega W(\cdot,\omega_0)\varphi_0 - \pa_\omega \pa_\varphi f(\cdot,\omega_0,0)\varphi_0=-\pa_\omega W(\cdot,\omega_0)\varphi_0.$$
By the closed range theorem is $(A-W(\cdot,\omega_0))u=-\pa_\omega W(\cdot,\omega_0)\varphi_0$ solvable if and only if (Wt) is violated. Here we have used the fact that $A-W(\cdot,\omega_0)$ is a Fredholm operator (in particular, $\text{Ran}(A)$ is closed), see \cite[Theorem IV.5.28]{Kato}. 
\end{remark}

\begin{remark}\textit{(Example of the potential and the nonlinearity.)}
Assumptions (A1)-(f4) are satisfied for instance by equation \eqref{Schrod} with $D(A)=H^2(\R,\C)$ provided (W1) and (Wt) hold and $\Gamma$ satisfies
$$\|\Gamma(\cdot,\omega_1)-\Gamma(\cdot,\omega_2)\| \leq L |\omega_1-\omega_2|$$
for some $L>0$ and all $\omega_{1,2}\in B_\delta(\omega_0)$. An example corresponding to the Drude model for metals (see Sec. \ref{Section_SPPs}) is 
$$W(x,\omega)=\frac{\omega^2}{c^2}\left(1-\frac{\omega_p^2}{\omega^2+\ri\gamma\omega}\right)-k^2, \quad\; \Gamma(x,\omega)=3\frac{\omega^2}{c^2}\hat{\chi}^{(3)}(x,\omega)$$ 
with parameters $\gamma,\omega_p,k\in\R$ and $\hat{\chi}^{(3)}$ a bounded function, Lipschitz continuous in $\omega$. 
This choice will be important for modelling SPPs. Note that, with such potential $W$, the operator $L(\cdot,\omega):=\partial_x^2+W(\cdot,\omega)$ is non-selfadjoint and also nonlinear in $\omega$.
\end{remark}
\begin{remark}
	Let us discuss the role of the parameter $\tau$ in the expansion \eqref{omega_varphi}. Clearly, $\tau$ determines the accuracy of the expansion given by the first two terms. According to Theorem \ref{Theorem_BIF_NLeigv}, if $\alpha\beta \leq 1$, then the optimal value is $\tau =\alpha\beta$, which is proportional to the difference of the degree of the lowest degree term in $f$ and the next term. Notice also that higher order terms in the nonlinearity do not play any role in the choice of $\tau$. To give an example, consider the nonlinearities in the table below.

\begin{table}[H]
	\centering
	\begin{tabular}{ccccc}
		\toprule
		$f(x)$ & $\alpha$ & $\beta$ & $\tau_{\max}$ & $\omega$\\
		\midrule
		$|\varphi|^2\varphi$ & $\frac12$ & $+\infty$ & $1$ &$\omega_0+\varepsilon\nu+\varepsilon^2\sigma$\\
		$|\varphi|^2\varphi+|\varphi|^4\varphi$ & $\frac12$ & $2$ & $1$ & $\omega_0+\varepsilon\nu+\varepsilon^2\sigma$\\
		$|\varphi|^2\varphi+|\varphi|^3\varphi$ & $\frac12$ & $1$ & $\frac12$ & $\omega_0+\varepsilon\nu+\varepsilon^{\frac32}\sigma$\\
		$|\varphi|^2\varphi+|\varphi|^3\varphi+|\varphi|^4\varphi$ & $\frac12$ & $1$ & $\frac12$ & $\omega_0+\varepsilon\nu+\varepsilon^{\frac32}\sigma$\\
		\bottomrule
	\end{tabular}
\end{table}

As explained in Sec. \ref{Section_SPPs}, in the applications of Theorem \ref{Theorem_BIF_NLeigv} to time harmonic electromagnetic waves, the relevant nonlinearities are odd. In the case of a cubic nonlinearity ($\alpha = 1/2$) the first correction term in $f$ is of the kind $|\varphi|^4\varphi$ and therefore we have $\beta=2$ such that we are in the optimal case $\tau=1$.
\end{remark}

	The strategy of the proof of Theorem \ref{Theorem_BIF_NLeigv} may be summarized as follows. We employ a Lyapunov-Schmidt reduction making use of spectral projections (here the simplicity of the eigenvalue $\omega_0$ is used) to decompose the problem into a system in which one rather easily determines $\nu$ and $\phi$. Then, the rest becomes a system for the unknowns $\sigma$ and $\psi$, which will be solved by means of a nested fixed-point argument. In particular, the assumption that $\omega_0$ be isolated is exploited to invert the operator $L(\cdot,\omega_0)$ restricted to $\langle\varphi_0^*\rangle^\perp$.

\begin{proof}[Proof of Theorem \ref{Theorem_BIF_NLeigv}]
$ $\newline
\textbf{Lyapunov-Schmidt decomposition.}

In this initial step we reformulate problem \eqref{eq} with the ansatz in \eqref{omega_varphi} as a system of two equations using the Lyapunov-Schmidt decomposition.

Let us first introduce the projections $P_0: u\mapsto\langle u,\varphi_0^*\rangle\varphi_0$ and $Q_0:=Id-P_0$. Clearly, $P_0:L^2(\R^d,\C)\to \langle \varphi_0\rangle$ and $Q_0:L^2(\R^d,\C)\to \langle \varphi_0^*\rangle^\perp$. Using our constraint $\langle \varphi,\varphi_0^*\rangle = \eps^\alpha$, it is easy to see that $P_0(\phi+\varepsilon^\tau\psi)=0$. Applying then $P_0$ to our equation \eqref{eq}, we get
	\begin{equation}\label{proj_P0_1}
	\begin{split}
	\langle L(\cdot,\omega)\varphi,\varphi_0^*\rangle&=\langle f(\cdot,\omega,\varphi),\varphi_0^*\rangle\\
	&=\langle f(\cdot,\omega,\varphi)-\varepsilon^{\alpha+1}g_{\omega_0}(\cdot)|\varphi_0|^\frac1\alpha\varphi_0,\varphi_0^*\rangle+\varepsilon^{\alpha+1}\langle g_{\omega_0}(\cdot)|\varphi_0|^\frac1\alpha\varphi_0,\varphi_0^*\rangle.
	\end{split}
	\end{equation}
	Notice that the first term is of higher-order in $\varepsilon$, mainly because of our assumption (f4), see the forthcoming computations \eqref{4_sys_I_split}-\eqref{4_sys_I_split_3}.

	On the other hand, Taylor expanding $W(\cdot,\omega)$ in $\omega_0$ up to order two and using assumption (W1), we have
	\begin{equation*}
	\begin{split}
	\langle L(\cdot,\omega)\varphi,\varphi_0^*\rangle&=\langle L(\cdot,\omega_0)\varphi,\varphi_0^*\rangle+\langle\left(L(\cdot,\omega)-L(\cdot,\omega_0)\right)\varphi,\varphi_0^*\rangle\\
	&=\langle\varphi,L(\cdot,\omega_0)^*\varphi_0^*\rangle-\langle\left(W(\cdot,\omega)-W(\cdot,\omega_0)\right)\varphi,\varphi_0^*\rangle\\
	&=-\left\langle\left(\partial_\omega W(\cdot,\omega_0)(\omega-\omega_0)+\frac12\partial_\omega^2W(\cdot,\omega_0)(\omega-\omega_0)^2+I(\cdot, \omega)\right)\varphi,\varphi_0^*\right\rangle,
	\end{split}
	\end{equation*}
	where
	\begin{equation*}
	I(x,\omega):=\frac{(\omega-\omega_0)^3}{2\pi \ri}\int_{\pa B_r(\omega_0)}\frac{W(x,z)}{(z-\omega_0)^3(z-\omega)}dz
	\end{equation*}
	for any $r<\delta$, see \eqref{E:Taylor-exp-W}.
	Inserting now the expansions of $\omega$ and $\varphi$ from \eqref{omega_varphi}, i.e.
	\begin{equation*}
	\omega=\omega_0+\varepsilon\nu+\varepsilon^{1+\tau}\sigma\qquad\mbox{and}\qquad\varphi=\varepsilon^\alpha\varphi_0+\varepsilon^{\alpha+1}\phi+\varepsilon^{\alpha+1+\tau}\psi,
	\end{equation*}
	we obtain
	\begin{equation}\label{proj_P0_2}
	\begin{split}
	-\langle L(\cdot,\omega)\varphi,\varphi_0^*\rangle&=\varepsilon^{\alpha+1}\nu\langle\partial_\omega W(\cdot,\omega_0)\varphi_0,\varphi_0^*\rangle+\varepsilon^{\alpha+1+\tau}\sigma\langle\partial_\omega W(\cdot,\omega_0)\varphi_0),\varphi_0^*\rangle\\
	&\quad+\varepsilon^{\alpha+2}\bigg\langle\bigg(\nu\partial_\omega W(\cdot,\omega_0)\phi+\frac{\nu^2}2\partial_\omega^2W(\cdot,\omega_0)\varphi_0\bigg),\varphi_0^*\bigg\rangle+\langle v(\cdot,\omega,\varphi),\varphi_0^*\rangle,
	\end{split}
	\end{equation}
	where $v$ collects all other terms of higher-order in $\varepsilon$, namely
	\begin{equation}\label{v}
	\begin{split}
	v(\cdot,\omega,\varphi):&=\varepsilon^{\alpha+2+\tau}\sigma\partial_\omega W(\cdot,\omega_0)\phi+\varepsilon^{\alpha+2+\tau}(\nu+\varepsilon^\tau\sigma)\partial_\omega W(\cdot,\omega_0)\psi+\varepsilon^{\alpha+3}\frac{\nu^2}2\partial_\omega^2W(\cdot,\omega_0)(\phi+\varepsilon^\tau\psi)\\
	&\quad+\varepsilon^{\alpha+2+\tau}\frac12\partial_\omega^2W(\cdot,\omega_0)(2\nu\sigma+\varepsilon^\tau\sigma^2)(\varphi_0+\varepsilon\phi+\varepsilon^{1+\tau}\psi)+\varepsilon^\alpha I(\cdot,\omega)(\varphi_0+\varepsilon\phi+\varepsilon^{1+\tau}\psi).
	\end{split}
	\end{equation}	
	Comparing now \eqref{proj_P0_1} and \eqref{proj_P0_2}, the terms of order $\alpha+1$ in $\varepsilon$ match if and only if we take
	\begin{equation}\label{nu_GEN}
	\nu:=-\frac{\langle g_{\omega_0}(\cdot)|\varphi_0|^\frac1\alpha\varphi_0,\varphi_0^*\rangle}{\langle\partial_\omega W(\cdot,\omega_0)\varphi_0,\varphi_0^*\rangle},
	\end{equation}
	which is well-defined thanks to assumption (Wt). From the rest of \eqref{proj_P0_1}-\eqref{proj_P0_2} we obtain
	\begin{equation}\label{3_sys}
	\begin{split}
	\varepsilon^{\alpha+1+\tau}\sigma\langle\partial_\omega W(\cdot,\omega_0)\varphi_0,\varphi_0^*\rangle=&-\varepsilon^{\alpha+2}\Big\langle\left(\nu\partial_\omega W(\cdot,\omega_0)\phi+\frac{\nu^2}2\partial_\omega^2W(\cdot,\omega_0)\varphi_0\right),\varphi_0^*\Big\rangle\\
	&-\langle v(\cdot,\omega,\varphi),\varphi_0^*\rangle-\langle f(\cdot,\omega,\varphi)-\varepsilon^{\alpha+1}g_{\omega_0}(\cdot)|\varphi_0|^\frac1\alpha\varphi_0,\varphi_0^*\rangle.
	\end{split}
	\end{equation}
	Let us now apply $Q_0$ to \eqref{eq}.
	On the one hand we have
	\begin{equation*}
	Q_0L(\cdot,\omega)Q_0\varphi
	=-Q_0\big(W(\cdot,\omega)-W(\cdot,\omega_0)\big)(\varepsilon^{\alpha+1}\phi+\varepsilon^{\alpha+1+\tau}\psi)+Q_0L(\cdot,\omega_0)(\varepsilon^{\alpha+1}\phi+\varepsilon^{\alpha+1+\tau}\psi)
	\end{equation*}
	and on the other hand
	\begin{equation*}
	\begin{split}
	Q_0L(\cdot,\omega)Q_0\varphi&=Q_0L(\cdot,\omega)\varphi-\varepsilon^\alpha Q_0\big(L(\cdot,\omega)-L(\cdot,\omega_0)\big)\varphi_0-\varepsilon^\alpha Q_0L(\cdot,\omega_0)\varphi_0\\
	&=Q_0f(\cdot,\omega,\varphi)+Q_0\big(W(\cdot,\omega)-W(\cdot,\omega_0)\big)(\varepsilon^\alpha\varphi_0).
	\end{split}
	\end{equation*}
	Therefore, we obtain
	\begin{equation*}
	Q_0L(\cdot,\omega_0)Q_0(\varepsilon^{\alpha+1}\phi+\varepsilon^{\alpha+1+\tau}\psi)=Q_0f(\cdot,\omega,\varphi)+\varepsilon^\alpha Q_0\big(W(\cdot,\omega)-W(\cdot,\omega_0)\big)(\varphi_0+\varepsilon\phi+\varepsilon^{1+\tau}\psi).
	\end{equation*}
	An expansion of the last term as in \eqref{proj_P0_2}-\eqref{v} yields
	\begin{equation}\label{proj_Q0_all}
	\begin{split}
	\varepsilon^{\alpha+1}Q_0L(\cdot&,\omega_0)Q_0\phi+\varepsilon^{\alpha+1+\tau}Q_0L(\cdot,\omega_0)Q_0\psi=\varepsilon^{\alpha+1}Q_0\big(g_{\omega_0}(\cdot)|\varphi_0|^\frac1\alpha\varphi_0\big)+Q_0\big(f(\cdot,\omega,\varphi)\\
	&-\varepsilon^{\alpha+1}g_{\omega_0}(\cdot)|\varphi_0|^\frac1\alpha\varphi_0\big)+Q_0\bigg(\varepsilon^{\alpha+1}\nu\partial_\omega W(\cdot,\omega_0)\varphi_0 +\varepsilon^{\alpha+2}\bigg(\frac{\nu^2}2\partial_\omega^2W(\cdot,\omega_0)\varphi_0\\
	&+\nu\partial_\omega W(\cdot,\omega_0)\phi\bigg)+\varepsilon^{\alpha+1+\tau}\sigma\partial_\omega W(\cdot,\omega_0)\varphi_0+v(\cdot,\omega,\varphi)\bigg).
	\end{split}
	\end{equation}
	Rewriting \eqref{nu_GEN} as
	\begin{equation*}
	\begin{split}
	g_{\omega_0}(\cdot)|\varphi_0|^\frac1\alpha\varphi_0-\langle g_{\omega_0}(\cdot)|\varphi_0|^\frac1\alpha\varphi_0,\varphi_0^*\rangle\varphi_0=g_{\omega_0}(\cdot)|\varphi_0|^\frac1\alpha\varphi_0+\nu\langle\partial_\omega W(\cdot,\omega_0)\varphi_0,\varphi_0^*\rangle\varphi_0
	\end{split}
	\end{equation*}
	and using \eqref{proj_Q0_all}, we obtain
	\begin{equation}\label{proj_Q0_all_better}
	\begin{split}
	\varepsilon^{\alpha+1}Q_0L(\cdot,\omega_0)Q_0\phi&+\varepsilon^{\alpha+1+\tau}Q_0L(\cdot,\omega_0)Q_0\psi=\varepsilon^{\alpha+1}\big(g_{\omega_0}(\cdot)|\varphi_0|^\frac1\alpha\varphi_0+\nu\partial_\omega W(\cdot,\omega_0)\varphi_0\big)\\
	&+Q_0\big(f(\cdot,\omega,\varphi)-\varepsilon^{\alpha+1}g_{\omega_0}(\cdot)|\varphi_0|^\frac1\alpha\varphi_0\big)+\varepsilon^{\alpha+1+\tau}\sigma Q_0\big(\partial_\omega W(\cdot,\omega_0)\varphi_0\big)\\
	&+\varepsilon^{\alpha+2}Q_0\bigg(\frac{\nu^2}2\partial_\omega^2W(\cdot,\omega_0)\varphi_0+\nu\partial_\omega W(\cdot,\omega_0)\phi\bigg)+Q_0\big(v(\cdot,\omega,\varphi)\big).
	\end{split}
	\end{equation}
	
	\noindent Again, imposing that the terms of the lowest-order in $\varepsilon$ match, we get a linear equation for $\phi$:
	\begin{equation}\label{phi_GEN}
	Q_0L(\cdot,\omega_0)Q_0\phi=g_{\omega_0}(\cdot)|\varphi_0|^\frac1\alpha\varphi_0+\nu\partial_\omega W(\cdot,\omega_0)\varphi_0.
	\end{equation}
	Notice that, with our choice of $\nu$, equation \eqref{phi_GEN} is uniquely solvable in $Q_0D(A)=D(A)\cap\langle\varphi_0^*\rangle^\perp$ by the closed range theorem. Indeed, the operator on the left hand side is Fredholm (see \cite[Theorem IV.5.28]{Kato}, where the fact that $0$ is a simple isolated eigenvalue of $L(\cdot,\omega_0)$, see (E1)-(E2), is used) and the right hand side is orthogonal to the kernel of the adjoint operator, i.e. to $\varphi_0^*$, due to \eqref{nu_GEN}.
	\vskip0.2truecm
	\noindent The rest of \eqref{proj_Q0_all_better} produces the following equation for $(\sigma,\psi)$:
	\begin{equation}\label{4_sys}
	\begin{split}
	\varepsilon^{\alpha+1+\tau}Q_0L(\cdot,\omega_0)Q_0\psi&=Q_0\big(f(\cdot,\omega,\varphi)-\varepsilon^{\alpha+1}g_{\omega_0}(\cdot)|\varphi_0|^\frac1\alpha\varphi_0\big)+\varepsilon^{\alpha+1+\tau}\sigma Q_0(\partial_\omega W(\cdot,\omega_0)\varphi_0)\\
	&\quad+\varepsilon^{\alpha+2}Q_0\bigg(\frac{\nu^2}2\partial_\omega^2W(\cdot,\omega_0)\varphi_0+\nu\partial_\omega W(\cdot,\omega_0)\phi\bigg)+Q_0(v(\cdot,\omega,\varphi))\\
	&=:R_1+\varepsilon^{\alpha+1+\tau}R_2+\varepsilon^{\alpha+2}R_3+R_4=:R(\sigma,\psi).
	\end{split}
	\end{equation}
	
 \noindent	
\textbf{Fixed Point Argument.}
	
	In order to solve our initial problem \eqref{eq}, we now need to solve system \eqref{3_sys}, \eqref{4_sys} for $(\sigma,\psi)\in\C\times D(A)$. Inserting then $\sigma$ and $\psi$ into \eqref{omega_varphi} produces a solution $(\omega,\varphi)$ of \eqref{eq}.
	
	We proceed by a fixed point argument. Note that although a direct fixed point argument for $(\sigma,\psi)$ is possible, we opt for a nested version, where we first solve for $\psi$ as a function of $\sigma$ and subsequently solve for $\sigma$. This approach is arguably more transparent.
	
	\vskip0.2truecm
	Let us first address equation \eqref{4_sys} and write it as a fixed point equation for $\psi$, exploiting our assumptions on the eigenvalue $\omega_0$, which is assumed simple and isolated. This actually means that $Q_0L(\omega_0)Q_0:Q_0D(A)\to Q_0L^2(\R^d)$ is boundedly invertible in $Q_0D(A)$ and its norm is bounded by a constant $C(\omega_0)$:
	\begin{equation}\label{4_sys_PTOF}
\begin{split}
	&\psi=\eps^{-(\alpha+1+\tau)}\big[Q_0L(\omega_0)Q_0\big]^{-1}R(\sigma,\psi):=G(\sigma,\psi),\\ 
	&\mbox{with}\qquad\|G(\sigma,\psi)\|_A\leq C(\omega_0)\eps^{-(\alpha+1+\tau)}\|R(\sigma,\psi)\|.
\end{split}
	\end{equation}
	In our nested fixed point argument for \eqref{3_sys}, \eqref{4_sys_PTOF} we first solve \eqref{4_sys_PTOF} for $\psi\in B_{r_2}(0)\subset D(A)$ for all $\sigma\in B_{r_1}(0)\subset\C$ fixed with $r_1>0$ arbitrary. To this aim we need to show that for each $\sigma\in B_{r_1}(0)$ there exists $r_2>0$ so that\footnotemark\footnotetext{With a little abuse of notation, henceforth we write $G(\psi):=G(\sigma,\psi)$ and similarly for $R$ when $\sigma$ is assumed to be fixed.}
	\begin{enumerate}
		\item[(\textit{i})] $\psi\in B_{r_2}(0)$ $\Rightarrow$ $G(\psi)\in B_{r_2}(0)$,
		\item[(\textit{ii})] $\exists \rho\in(0,1):\ \|G(\psi_1)-G(\psi_2)\|_A\leq\rho\|\psi_1-\psi_2\|_A$ for all $\psi_1,\psi_2\in B_{r_2}(0)$
	\end{enumerate}
	if $\varepsilon >0$ is small enough. 
	
	Then, having obtained $\psi=\psi(\sigma)\in B_{r_2}(0)$, we shall solve equation \eqref{3_sys} for $\sigma$ and finally find a suitable $r_1$. Note that the fixed point argument for equation \eqref{3_sys} requires the Lipschitz continuity of $\sigma\mapsto\psi(\sigma)$, which we verify below.

	To ensure (\textit{i}), we need to estimate $\|R(\psi)\|$. The second and the third term in \eqref{4_sys} are easy to handle. Henceforth, we track the dependence of all constants on $\sigma$ and $\psi$ via $r_1,r_2$.
	\begin{equation}\label{4_sys_II}
	\|R_2\|\leq|\sigma|\|\partial_\omega W(\cdot,\omega_0)\|_\infty\leq Cr_1,
	\end{equation}
	\begin{equation}\label{4_sys_III}
	\|R_3\|\leq\max\left\{\tfrac{|\nu|^2}2,\nu\|\phi\|\right\}\left(\|\partial_\omega W(\cdot,\omega_0)\|_\infty+\|\partial^2_\omega W(\cdot,\omega_0)\|_\infty\right)\leq C
	\end{equation}
using \eqref{W-der-est}.	Let us now deal with $R_4$. Inspecting \eqref{v}, we obtain
	\begin{equation}\label{4_sys_IV}
	\begin{split}
	\|R_4\|&\leq \|v(\cdot,\omega,\varphi)\|\leq C\left[\varepsilon^{\alpha+2+\tau}(r_1+(1+\varepsilon^\tau r_1)\|\psi\|)+\varepsilon^{\alpha+3}(1+\varepsilon^\tau\|\psi\|)\right.\\
	&\left.\quad+\varepsilon^{\alpha+2+\tau}r_1(1+\varepsilon^\tau r_1)(1+\varepsilon^{1+\tau}\|\psi\|)+\varepsilon^{\alpha+3}(1+\eps^\tau r_1)(1+\varepsilon^{1+\tau}\|\psi\|)\right]\\
	&\leq C\left[\varepsilon^{\alpha+2+\tau}(r_1+\|\psi\|)+\varepsilon^{\alpha+3}\right]+\varepsilon^{\alpha+2+2\tau} h_1(r_1,\|\psi\|),
	\end{split}
	\end{equation}
	where $h_1$ is polynomial in $r_1$, linear in $\|\psi\|$ and satisfies $h_1(0,\|\psi\|)=0.$
	Actually, all terms appearing in \eqref{v} are easy to estimate, so here we just briefly justify the one for the integral rest $I(x,\omega)$, for later use too. Indeed, using \eqref{E:Tn-est} for $n=3$, we 
	have 
	\begin{equation}\label{I_omega_estimate}
	\begin{split}
	\|I(\cdot,\omega)\|_\infty&\leq \frac{2M}{r^3}|\omega-\omega_0|^3=\frac{2M}{r^3}\varepsilon^3|\nu+\varepsilon^\tau\sigma|^3\leq \frac{2MC}{r^3}\varepsilon^3(1+\varepsilon^\tau r_1)
	\end{split}
	\end{equation}
	for any $0<r<\delta$ and all $\omega\in B_{r/2}(\omega_0)$, i.e. for all $\eps>0$ small enough. Recall that $M=\sup\limits_{\omega\in B_\delta(\omega_0)}\|W(\cdot,\omega)\|_\infty.$
	Finally, we need to estimate $R_1$, which involves the nonlinearity. We split it as
	\begin{align}\label{4_sys_I_split}
	f(\cdot,\omega,\varphi)-\varepsilon^{\alpha+1}g_{\omega_0}(\cdot)|\varphi_0|^{\frac1\alpha}\varphi_0&=\left(f(\cdot,\omega,\varphi)-g_\omega(\cdot)|\varphi|^{\frac1\alpha}\varphi\right)+g_\omega(\cdot)\!\left(|\varphi|^{\frac1\alpha}\varphi-\varepsilon^{\alpha+1}|\varphi_0|^{\frac1\alpha}\varphi_0\right)\notag\\
	&\quad+\varepsilon^{\alpha+1}\left(g_\omega(\cdot)-g_{\omega_0}(\cdot)\right)|\varphi_0|^{\frac1\alpha}\varphi_0
	\end{align}
	and estimate term by term. First, by (f4) one gets
	\begin{equation}\label{4_sys_I_split_1}
	\begin{split}
	\|f(\cdot,\omega,\varphi)-g_\omega|\varphi|^{\frac1\alpha}\varphi\|&\leq C(\omega_0)\||\varphi|^{\frac1\alpha+1+\beta}\|\leq C\|\varphi\|_\infty^{\frac1\alpha+\beta}\|\varphi\|\leq C\|\varphi\|_A^{\frac1\alpha+1+\beta}\\
	&\leq C\eps^{\alpha+1+\alpha\beta}\|\varphi_0+\varepsilon\phi+\varepsilon^{1+\tau}\psi\|_A^{\frac1\alpha+1+\beta}\\
	&\leq C\eps^{\alpha+1+\alpha\beta}(1+\varepsilon^{1+\tau}r_2)^{\tfrac{1}{\alpha}+\beta}(1+\varepsilon^{1+\tau}\|\psi\|_A)\\
	&\leq C_1(r_2)\eps^{\alpha+1+\alpha\beta}(1+\varepsilon^{1+\tau}\|\psi\|_A)	
	\end{split}
	\end{equation}
	for $\eps>0$ small enough. In equation \eqref{4_sys_I_split_1} we used the embedding $D(A)\hookrightarrow L^\infty(\R^d)$. The dependence $r_2\mapsto C_1(r_2)$ is of power type; in detail, 
	$$C_1(r_2)=C_0(1+\eps^{1+\tau}r_2)^{1/\alpha+\beta}=C_0\left(1+\left(\tfrac1\alpha+\beta\right)\eps^{1+\tau}r_2 (1+\eps^{1+\tau}r_*)^{1/\alpha+\beta-1}\right)$$
	for some $r_*\in [0,r_2]$. For each $r_2>0$ there exists $\eps_0=\eps_0(r_2)$ such that $$(\tfrac{1}{\alpha}+\beta)\eps^{1+\tau}r_2 (1+\eps^{1+\tau}r_*)^{1/\alpha+\beta-1}\leq 1$$ 
	for all $\eps\in (0,\eps_0).$
	In conclusion
	\begin{equation}\label{E:C1_est}
	C_1(r_2)\leq 2C_0=2C_1(0)
	\end{equation}
	for all $\eps\in (0,\eps_0).$
	
	Next,
	\begin{align}\label{4_sys_I_split_2}
	\left\|g_\omega\left(|\varphi|^{\frac1\alpha}\varphi-\varepsilon^{\alpha+1}|\varphi_0|^{\frac1\alpha}\varphi_0\right)\right\|&\leq\|g_\omega\|_\infty\varepsilon^{\alpha+1}\||\varphi_0+\varepsilon\phi+\varepsilon^{1+\tau}\psi|^{\frac1\alpha}(\varphi_0+\varepsilon\phi+\varepsilon^{1+\tau}\psi)-|\varphi_0|^{\frac1\alpha}\varphi_0\|\notag\\
	&\leq C\varepsilon^{\alpha+1}\|\varepsilon\phi+\varepsilon^{1+\tau}\psi\|\leq C\varepsilon^{\alpha+2}(1+\varepsilon^\tau\|\psi\|),
	\end{align}
	where we used the fact that the map $z\mapsto|z|^\frac1\alpha z$ is locally Lipschitz as well as estimates of the type $\||\varphi|^{1/\alpha}\phi\|\leq \|\varphi\|^{1/\alpha}_{\infty}\|\phi\|\leq c\|\varphi\|^{1/\alpha}_{D(A)}\|\phi\|=C$. Finally, since $g_\omega$ is Lipschitz in $\omega$ by (f4),
	\begin{equation}\label{4_sys_I_split_3}
	\begin{split}
	\varepsilon^{\alpha+1}\|\left(g_\omega-g_{\omega_0}\right)|\varphi_0|^{\frac1\alpha}\varphi_0\|&\leq \varepsilon^{\alpha+1}C\|g_\omega-g_{\omega_0}\|_\infty\leq \varepsilon^{\alpha+1}CK_g|\omega-\omega_0|\\
	&\leq C\varepsilon^{\alpha+2}|\nu+\varepsilon^\tau\sigma|\leq C\varepsilon^{\alpha+2}(1+\varepsilon^\tau r_1).
	\end{split}
	\end{equation}
	Hence from \eqref{4_sys_I_split}-\eqref{4_sys_I_split_3} we infer
	\begin{equation}\label{4_sys_I}
	\|R_1\|\leq C_1(r_2)\eps^{\alpha+1+\alpha\beta}(1+\eps^{1+\tau}\|\psi\|_A)+C\eps^{\alpha+2}(1+\eps^{\tau}(r_1+\|\psi\|))
	\end{equation}
	and therefore, from \eqref{4_sys_II}-\eqref{4_sys_IV} and \eqref{4_sys_I}, that
	\begin{equation}\label{bound_G}
	\begin{split}
	\|G(\psi)\|_A&\leq C\varepsilon^{-(\alpha+1+\tau)}\|R(\psi)\|\\
	&\leq C_1(r_2)\eps^{\alpha\beta-\tau}(1+\varepsilon^{1+\tau}\|\psi\|_A)+C(\varepsilon^{1-\tau}+\eps\|\psi\|+r_1\big)+\varepsilon^{1+\tau}h_1(r_1,\|\psi\|).
	\end{split}
	\end{equation}
Using \eqref{E:C1_est} and $\tau\leq \alpha\beta$, we may further estimate
	\begin{equation*}
	\begin{split}
	\|G(\psi)\|_A\leq &2 C_0 (1+\eps^{1+\tau} r_2) + C(\eps^{1-\tau}+\eps r_2+r_1)+\eps^{1+\tau}h_1(r_1,r_2)\\
	\leq &4C_0 + \tilde{C}r_1
	\end{split}
	\end{equation*}
for $\eps>0$ small enough. This follows because for $\eps$ small enough we have $\eps^{1+\tau}h_1(r_1,r_2) <C_0$. Setting now $r_2=r_2(r_1) := 4C_0 + \tilde{C}r_1$, we get (\textit{i}) provided $\eps\in (0,\overline{\eps}_0)$ with some $\overline{\eps}_0=\overline{\eps}_0(r_1)>0.$

	Let us now address the contraction property (\textit{ii}). We take $\psi_1,\psi_2\in B_{r_2}(0)$ and define $\varphi_{1,2}:=\varepsilon^\alpha(\varphi_0+\varepsilon\phi+\varepsilon^{1+\tau}\psi_{1,2})$. By \eqref{4_sys_PTOF}, we need to estimate $\|R(\psi_1)-R(\psi_2)\|$, with $R$ defined in \eqref{4_sys}. First notice that $R_2$ and $R_3$ do not depend on $\psi$, so they will vanish in the difference and we have to handle just the nonlinear term and $v$.
	\begin{equation}\label{v_Lip_psi}
	\begin{split}
	\|v(\cdot,\omega,\varphi_1)-v(\cdot,\omega,\varphi_2)\|&\leq\varepsilon^{\alpha+2+\tau}C(1+\varepsilon^\tau r_1)\|\psi_1-\psi_2\|+\varepsilon^{\alpha+3+\tau}\|\psi_1-\psi_2\|\\
	&\quad+C\varepsilon^{\alpha+3+2\tau}(r_1+\varepsilon^\tau r_1^2)\|\psi_1-\psi_2\| +C\varepsilon^{\alpha+4+\tau}(1+\varepsilon^\tau r_1)\|\psi_1-\psi_2\|\\
	&\leq\varepsilon^{\alpha+2+\tau}C_2(r_1)\|\psi_1-\psi_2\|,
	\end{split}
	\end{equation}	
	where $C_2$ is polynomial in $r_1$.	Here, estimate \eqref{I_omega_estimate} was used. Next, by (f3),
	\begin{equation}\label{f_Lip_psi}
	\begin{split}
	\|f(\cdot,\omega,\varphi_1)-f(\cdot,\omega,\varphi_2)\|&=K_f^\varphi(\omega_0)\varepsilon^{\alpha+1}\|\varphi_0+\varepsilon\phi+\varepsilon^{1+\tau}\psi_1-(\varphi_0+\varepsilon\phi+\varepsilon^{1+\tau}\psi_2)\|\\
	&\leq K_f^\varphi(\omega_0)\varepsilon^{\alpha+2+\tau}\|\psi_1-\psi_2\|.
	\end{split}
	\end{equation}
	Hence, by \eqref{4_sys_PTOF}, \eqref{v_Lip_psi} and \eqref{f_Lip_psi}, we get
	\begin{equation*}
	\|G(\psi_1)-G(\psi_2)\|_A\leq\varepsilon^{-(\alpha+1+\tau)}C\|R(\psi_1)-R(\psi_2)\|\leq\varepsilon \max\{K_f^\varphi(\omega_0),C_2(r_1)\}\|\psi_1-\psi_2\|,
	\end{equation*}
	which yields (\textit{ii}) if $\varepsilon$ is small enough. Therefore, applying Banach's fixed point theorem, we infer the existence of a solution $\psi$ of equation \eqref{4_sys}. More precisely, for any $r_1>0$ and for any $\sigma\in B_{r_1}(0)\subset\C$ there exists $r_2=r_2(r_1)$ and $\eps_0>0$ such that for each $\eps\in (0,\eps_0)$ there is a unique $\psi=\psi(\sigma)\in B_{r_2}(0)\subset D(A)$ such that $(\sigma,\psi)$ solves \eqref{4_sys}.
	\vskip0.2truecm
	Let us now address equation \eqref{3_sys}. Inserting $\psi=\psi(\sigma)$ and dividing by the factor $\varepsilon^{\alpha+1+\tau}\langle\partial_\omega W(\cdot,\omega_0)\varphi_0,\varphi_0^*\rangle$, which is nonzero by (Wt), this becomes a fixed point equation for $\sigma$:
	\begin{equation}\label{3_sys_PTO}\tag{\ref{3_sys}'}
	\begin{split}
	\sigma&=\left(\varepsilon^{\alpha+1+\tau}\langle\partial_\omega W(\cdot,\omega_0)\varphi_0,\varphi_0^*\rangle\right)^{-1}\!\bigg(-\varepsilon^{\alpha+2}\Big\langle\left(\nu\partial_\omega W(\cdot,\omega_0)\phi+\frac{\nu^2}2\partial_\omega^2W(\cdot,\omega_0)\varphi_0\right),\varphi_0^*\Big\rangle\\
	&\quad-\langle v(\cdot,\omega,\varphi(\sigma)),\varphi_0^*\rangle-\langle f(\cdot,\omega,\varphi(\sigma))-\varepsilon^{\alpha+1}g_{\omega_0}(\cdot)|\varphi_0|^\frac1\alpha\varphi_0,\varphi_0^*\rangle\bigg)\\
	&=: S(\sigma),
	\end{split}
	\end{equation}
	where $\varphi(\sigma)$ is given by \eqref{omega_varphi} with $\psi=\psi(\sigma).$
	We need to show that for some $r_1>0$
	\begin{enumerate}
		\item[(\textit{i'})] $S:B_{r_1}(0)\to B_{r_1}(0)$,
		\item[(\textit{ii'})] there exists $\rho\in(0,1)$ so that $|S(\sigma_1)-S(\sigma_2)|\leq\rho |\sigma_1-\sigma_2|$ for all $\sigma_1,\sigma_2\in B_{r_1}(0)$
	\end{enumerate}
	if $\varepsilon>0$ is small enough.
	
	Notice that the first term on the right-hand side of (\ref{3_sys}') is independent of $\sigma$ and $\psi$, therefore its norm can be simply estimated by a constant. Decomposing the nonlinear term as in \eqref{4_sys_I_split}, according to estimates \eqref{4_sys_I_split_1}-\eqref{4_sys_I_split_3}, we get
	\begin{equation}\label{3_sys_PTO_NL}
	\|f(\cdot,\omega,\varphi)-\varepsilon^{\alpha+1}g_{\omega_0}(\cdot)|\varphi_0|^{\frac1\alpha}\varphi_0\|\leq C_1(r_2)\varepsilon^{\alpha+1+\alpha\beta}(1+\varepsilon^{1+\tau}r_2)+C\varepsilon^{\alpha+2}(1+\varepsilon^\tau(|\sigma|+r_2)).
	\end{equation}
	Moreover, similarly to \eqref{4_sys_IV} we have
	\begin{equation}\label{3_sys_PTO_L}
	\begin{split}
	\|v(\cdot,\omega,\varphi)\|&\leq\varepsilon^{\alpha+2+\tau}C|\sigma|+\varepsilon^{\alpha+2+\tau}(1+\varepsilon^\tau|\sigma|)r_2+C\varepsilon^{\alpha+3+\tau}r_2\\
	&\quad+\varepsilon^{\alpha+2+\tau}C(|\sigma|+\varepsilon^\tau|\sigma|^2)(1+\varepsilon^{1+\tau}r_2)+\varepsilon^{\alpha+3}C(1+\varepsilon^\tau|\sigma|)(1+\varepsilon^{1+\tau}r_2)\\
	&\leq C\left(\varepsilon^{\alpha+2+\tau}(|\sigma|+r_2)+\varepsilon^{\alpha+3}\right)+\varepsilon^{\alpha+2+2\tau}h_1(|\sigma|,r_2).
	\end{split}
	\end{equation}
Therefore, combining \eqref{3_sys_PTO}, \eqref{3_sys_PTO_NL} and \eqref{3_sys_PTO_L}, we obtain
	\begin{equation*}
	\begin{split}
	|S(\sigma)|&\leq C_1(r_2)\varepsilon^{\alpha\beta-\tau}(1+\varepsilon^{1+\tau}r_2)+C\varepsilon^{1-\tau}(1+\varepsilon^\tau(|\sigma|+r_2))\\
	&\quad+C(\varepsilon(|\sigma|+r_2)+\varepsilon^{2-\tau})+\varepsilon^{1+\tau} h_1(|\sigma|,r_2).
	\end{split}
	\end{equation*}
	For each $r_1>0$ there is $\tilde{\eps}_0=\tilde{\eps}_0(r_1)>0$ such that
	$$|S(\sigma)|\leq \eps^{\min\{\alpha\beta-\tau,1-\tau\}}\left(2C_1(r_2)+2C\right)+1$$
	for all $|\sigma|<r_1$ and all $\eps\in (0,\tilde{\eps}_0)$. Recalling that $r_2=r_2(r_1)=4C_0+\tilde{C}r_1$ and that $\tau\leq\min\{1,\alpha\beta\}$, this implies that 
	$$|S(\sigma)|\leq 4C_0 + 2C+1\quad\;\text{for all } |\sigma|<r_1 \text{ and } \eps>0  \text{ small enough},$$
	i.e. $\eps\in (0,\eps_0)$ with $\eps_0=\eps_0(r_1).$ Setting $r_1:=4C_0+2C+1$, we have property (\textit{i}').

Finally, let us address the contraction property (\textit{ii}'). For $\sigma_{1,2}\in B_{r_1}(0)$  we define $\psi_{1,2}=\psi(\sigma_{1,2})$ and analogously $\omega_{1,2}$ and $\varphi_{1,2}$. We need to estimate $|S(\sigma_1)-S(\sigma_2)|$. Recalling the form of $v$ in \eqref{v}, 
	we get
	\begin{equation}\label{v_Lip_sigma}
	\begin{split}
	\|v(\cdot,\omega_1,\varphi_1)-v(\cdot,\omega_2,\varphi_2)\|&\leq\,C\bigg[\varepsilon^{\alpha+2+\tau}\big(|\sigma_1-\sigma_2|+\|\psi_1-\psi_2\|+\eps^\tau\|\sigma_1\psi_1-\sigma_2\psi_2\|\big)\\
	&\quad+\varepsilon^{\alpha+3+\tau}\|\psi_1-\psi_2\|+\varepsilon^{\alpha+2+\tau}\big(|\sigma_1^2-\sigma_2^2|+\varepsilon^{1+\tau}\|\sigma_1^2\psi_1-\sigma_2^2\psi_2\|\big)\\
	&\quad+\varepsilon^\alpha\|I(\cdot,\omega_1)-I(\cdot,\omega_2)\|_\infty+\varepsilon^{\alpha+1+\tau}\|I(\cdot,\omega_2)\|_\infty\|\psi_1-\psi_2\|\bigg]\\
	&\leq C_3(r_1)(1+r_2)\,\varepsilon^{\alpha+2+\tau}\big[|\sigma_1-\sigma_2|+\|\psi_1-\psi_2\|\big]
	\end{split}
	\end{equation}
	with $C_3(r_1)$ cubic in $r_1$. Here the estimates are rather standard and quite similar to the ones used in \eqref{3_sys_PTO_L}, so we just point out how to deal with the rest term $I$. We have, as in Remark \ref{R:Taylor}, 
	\begin{equation}\label{I_omega_CdV}
	\begin{split}
	I(\cdot,\omega_1)-I(\cdot,\omega_2)&=\frac1{2\pi\ri}\bigg((\omega_1-\omega_0)^3\int_{\pa B_r(\omega_0)}\frac{W(\cdot,z)}{(z-\omega_0)^3(z-\omega_1)}dz\\
	&\quad-(\omega_2-\omega_0)^3\int_{\pa B_r(\omega_0)}\frac{W(\cdot,z)}{(z-\omega_0)^3(z-\omega_2)}dz\bigg).
	\end{split}
	\end{equation}
	First, we estimate
	\begin{equation}\label{I_omega_stima_1}
	\begin{split}
	\frac{1}{2\pi}&\bigg\|\int_{\pa B_r(\omega_0)}\frac{W(\cdot,z)}{(z-\omega_0)^3(z-\omega_1)}dz\left((\omega_1-\omega_0)^3-(\omega_2-\omega_0)^3\right)\bigg\|_\infty\\
	&\leq \eps^3\frac{2M}{r^3}|(\nu+\eps^\tau \sigma_1)^3-(\nu+\eps^\tau \sigma_2)^3|\\
	&\leq \eps^3\frac{2M}{r^3}|\eps^\tau\nu^2(\sigma_1-\sigma_2)+3\eps^{2\tau}\nu(\sigma_1^2-\sigma_2^2) + \eps^{3\tau}(\sigma_1^3-\sigma_2^3)|\\
	&\leq C_4(r_1)\eps^{3+\tau}|\sigma_1-\sigma_2|,
	\end{split}
	\end{equation}
which holds for all $\eps>0$ small enough with $C_4(r_1)$ quadratic in $r_1$ using an estimate analogous to \eqref{I_omega_estimate}.	Second, we have
	\begin{align}
	&\frac{1}{2\pi}\left\|(\omega_2-\omega_0)^3\left(\int_{\pa B_r(\omega_0)}\frac{W(\cdot,z)}{(z-\omega_0)^3(z-\omega_1)}-\frac{W(\cdot,z)}{(z-\omega_0)^3(z-\omega_2)} dz\right)\right\|_\infty \nonumber\\
	&\quad \leq c\,\eps^3|\nu+\eps^\tau\sigma_2|^3\left\|\int_{\pa B_r(\omega_0)} \frac{W(\cdot,z)(\omega_1-\omega_2)}{(z-\omega_0)^3(z-\omega_1)(z-\omega_2)}dz \right\|_\infty \nonumber\\
	&\quad \leq \tilde{C}_5(r_1)\,\eps^{4+\tau}|\sigma_1-\sigma_2|\bigg\|\int_{\pa B_r(\omega_0)} \frac{W(\cdot,z)}{(z-\omega_0)^4(z-\omega_2)}dz \bigg\|_\infty \nonumber \\
	&\quad \leq C_5(r_1)\,\eps^{4+\tau}|\sigma_1-\sigma_2|\label{I_omega_stima_2}
	\end{align}
	for $\eps>0$ small enough, where in the second step we have used  $\omega_1-\omega_2=\eps^{1+\tau}(\sigma_1-\sigma_2)$ and estimated $|z-\omega_1|\geq \tfrac{1}{2}|z-\omega_0|$ for all $z\in \pa B_r(\omega_0)$, which holds for $\eps>0$ small enough. In the last step  estimate \eqref{E:Tn-est} was used again.
The constants $C_5(r_1),\tilde{C}_5(r_1)$ are cubic in $r_1$.
	Consequently, by \eqref{I_omega_CdV}-\eqref{I_omega_stima_2} we get
	$$\|I(\cdot,\omega_1)-I(\cdot,\omega_2)\|_\infty\leq\varepsilon^{3+\tau} C_6(r_1)|\sigma_1-\sigma_2|$$
	with $C_6(r_1)$ cubic in $r_1$.

	Now we have to deal with the third term in \eqref{3_sys_PTO} involving the nonlinearity. However, this is easily estimated using its Lipschitz behaviour in $\omega$ and $\varphi$ as in (f2)-(f3). Indeed,
	\begin{equation}\label{f_Lip_sigma}
	\begin{split}
	\|f(\cdot,\omega_1,\varphi_1)-f(\cdot,\omega_2,\varphi_2)\|&\leq\|f(\cdot,\omega_1,\varphi_1)-f(\cdot,\omega_1,\varphi_2)\|+\|f(\cdot,\omega_1,\varphi_2)-f(\cdot,\omega_2,\varphi_2)\|,
	\end{split}
	\end{equation}
	where the first term is estimated as in \eqref{f_Lip_psi}, whereas
	\begin{equation}\label{f_Lip_sigma_2}
	\|f(\cdot,\omega_1,\varphi_2)-f(\cdot,\omega_2,\varphi_2)\|
	\leq\varepsilon^{\alpha+1}K_f^\omega(\varphi_0)|\omega_1-\omega_2|=\varepsilon^{\alpha+2+\tau}K_f^\omega(\varphi_0)|\sigma_1-\sigma_2|.
	\end{equation}
	Therefore, combining \eqref{v_Lip_sigma} and \eqref{f_Lip_sigma}-\eqref{f_Lip_sigma_2}, we infer
	\begin{equation}\label{sigma_Lip_sigma}
	|S(\sigma_1)-S(\sigma_2)|\leq C_3(r_1)(1+r_2)\varepsilon(\|\psi_1-\psi_2\|+|\sigma_1-\sigma_2|)
	\end{equation}
	with $C_3(r_1)$ cubic in $r_1$.	Hence, it remains to show now that the map $\sigma\mapsto\psi(\sigma)$ is Lipschitz continuous. Taking $|\sigma_{1,2}|\leq r_1$, we shall estimate the difference $\|\psi_1-\psi_2\|$ starting from the fixed-point equation \eqref{4_sys_PTOF} for $\psi$, where as before we define $\psi_{1,2}:=\psi(\sigma_{1,2})$ and similarly for $\omega_{1,2}$ and $\varphi_{1,2}$. Indeed, exploiting the above estimates \eqref{v_Lip_sigma} and \eqref{f_Lip_sigma}-\eqref{f_Lip_sigma_2}, we have
	\begin{equation}\label{R_Lip_sigma}
	\begin{split}
	\|G(\sigma_1,\psi_1)-G(\sigma_2,\psi_2)\|&\leq C\varepsilon^{-(\alpha+1+\tau)}\big[\|f(\cdot,\omega_1,\varphi_1)-f(\cdot,\omega_2,\varphi_2)\|+C\varepsilon^{\alpha+1+\tau}|\sigma_1-\sigma_2|\\
	&\quad+\|v(\cdot,\sigma_1,\varphi_1)-v(\cdot,\sigma_2,\varphi_2)\|\big]\\
	&\leq C\left[(1+\eps K_f^\omega(\varphi_0))|\sigma_1-\sigma_2|+ C_2(r_1)(1+r_2)\eps(|\sigma_1-\sigma_2| + \|\psi_1-\psi_2\|)\right]\\
	&\leq C\left[2|\sigma_1-\sigma_2|+ C_2(r_1)(1+r_2)\eps\|\psi_1-\psi_2\|\right]
	\end{split}
	\end{equation}
	if $\eps>0$ is small enough.
	
	This, together with \eqref{4_sys_PTOF}, yields $\|\psi(\sigma_1)-\psi(\sigma_2)\|_A\leq 3C|\sigma_1-\sigma_2|$ for $\sigma_{1,2}\in B_{r_1}(0)$ and $\eps>0$ small enough. As a consequence we may conclude the fixed point argument for $\sigma$ because from \eqref{sigma_Lip_sigma} and \eqref{R_Lip_sigma} and from the fact that $r_2=r_2(r_1)$ we obtain
	\begin{equation*}
	|S(\sigma_1)-S(\sigma_2)|\leq \tilde{C}(r_1)\varepsilon|\sigma_1-\sigma_2|
	\end{equation*}
	with $\tilde{C}(r_1)>0$, which is the desired contraction property for suitably small values of $\varepsilon$. Hence, for small $\varepsilon>0$ the fixed point argument yields the sought solution for the system \eqref{3_sys}-\eqref{4_sys} and the proof of Theorem \ref{Theorem_BIF_NLeigv} is complete.	
\end{proof}


\section{Bifurcation of real nonlinear eigenvalues: the $\PT$-symmetric case}\label{S:PT}
In applications one often starts from a real eigenvalue of the linear problem and seeks a bifurcation branch $(\omega,\varphi)$ where the realness of the nonlinear eigenvalue is preserved, i.e. $\omega\in\R$. This is required also for the applications to SPPs in Sec. \ref{Section_SPPs}. However, the solution $\omega$ obtained by Theorem \ref{Theorem_BIF_NLeigv} is a-priori just complex. In this section we provide a symmetric situation in which the realness of $\omega$ is preserved in the bifurcation. This is based on \cite[Section III]{DS}, whereby we adapt the analysis for our more general context. For the sake of completeness and since Section \ref{Section_SPPs} is based on these results, we present most of the details here again.  
\begin{defn}\label{PTsymm}
	A function $\psi:\R^d\to\C$ ($d\geq1$) is called $\PT$\textit{-symmetric} if $\B\psi(x):=\overline{\psi(-x)}=\psi(x)$ for all $x\in\R^d$. Moreover, an operator $L$ acting on a Hilbert space with domain $D(L)$ is $\PT$\textit{-symmetric} when it commutes with $\B$, namely $L\B=\B L$ in $D(L)$.
\end{defn}
The above operator $\B$ is in fact the composition of the operator $\mathcal P$, the space reflection (parity), and $\mathcal T$, the complex conjugation, which corresponds to the time-reversal in quantum mechanics.

Notice that the Schr\"odinger operator $-\Delta+W$ with a complex potential $W$ is $\PT$-symmetric if and only if the real and the imaginary parts of $W$ satisfy respectively $\Real W(-x)=\Real W(x)$ and $\Imag W(-x)=-\Imag W(x)$ for all $x$. Moreover, general polynomial nonlinearities 
$$f(x,\varphi)=\sum_{p,q=0}^Na_{pq}(x)\varphi^p\overline{\varphi}^q$$
are $\PT$-symmetric if and only if the coefficients are so: $\overline{a_{pq}}(-x)=a_{pq}(x)$, see \cite[Sections III-IV]{DS}. An example is $f(\varphi)=|\varphi|^{2m+1}\varphi, m \in\N$. 

\begin{prop}\label{Thm_PT}
	Let $A$, $W$, $f$ and $g_\omega$ be as in Theorem \ref{Theorem_BIF_NLeigv} and suppose that they are $\PT$-symmetric. If $\omega_0\in\R$ is an algebraically simple eigenvalue, then for all $\varepsilon\in(0,\varepsilon_0)$ the nonlinear eigenpair $(\omega,\varphi)$ from Theorem \ref{Theorem_BIF_NLeigv} satisfies $\omega\in\R$ and $\B\varphi=\varphi$.
\end{prop}
\begin{remark}
	Under the assumptions on $A,W,f$, and $g_\omega$ as in Prop. \ref{Thm_PT} and under the simplicity assumption on $\omega_0$, the eigenfunction $\varphi_0$ may always be chosen $\PT$-symmetric. Indeed, $L(\cdot,\omega_0)=A+W(\cdot,\omega_0)$ is $\PT$-symmetric and, applying $\B$ to $L(\cdot,\omega_0)\varphi_0=0$, we get $L(\cdot,\omega_0)\big(\B\varphi_0)=0$ and we conclude by the simplicity of $\omega_0$.

Similarly, we obtain $\B\varphi_0^*=\varphi_0^*$ because $\B L^*(\cdot,\overline{\omega_0})=L^*(\cdot,\overline{\omega_0}) \B$. Indeed,
$$
\begin{aligned}
\langle v, L^*\phi\rangle&=\langle Lv, \phi\rangle=\overline{\langle \overline{(Lv)}(-\cdot),\overline{\phi}(-\cdot)\rangle}=\overline{\langle \B Lv,\B\phi\rangle} = \overline{\langle L\B v,\B\phi\rangle}\\
&=\overline{\langle \B v,L^*\B\phi\rangle}=\langle v,\overline{L^*}(-\cdot)\phi\rangle=\langle v, \B L^*\phi\rangle
\end{aligned}
$$
for all $v,\phi \in D(A)$.
\end{remark}
\begin{proof}
	According to our expansions \eqref{omega_varphi}, we need to prove that $\nu,\sigma$ are real and that $\phi,\psi$ are $\PT$-symmetric. First, $\nu$ in \eqref{nu_GEN} satisfies
	$P_0(\nu\partial_\omega W(\cdot,\omega_0)\varphi_0)=-P_0(g_{\omega_0}(\cdot)|\varphi_0|^\frac1\alpha\varphi_0)$, where we recall that $P_0$ is the spectral projection onto the eigenspace $\langle \varphi_0\rangle$ and that with our assumptions $P_0$ commutes with $\B$. Indeed, 
	\begin{equation*}
	\begin{split}
	\B P_0u&=\B(\langle u,\varphi_0^*\rangle\varphi_0)=\overline{\langle u,\varphi_0^*\rangle}\B\varphi_0=\langle \bar u,\overline{\varphi_0^*}\rangle\varphi_0=\langle\bar u,\varphi_0^*(-\cdot)\rangle\varphi_0\\
	&=\langle\overline{ u(-\cdot)},\varphi_0^*\rangle\varphi_0=\langle\B u,\varphi_0^*\rangle\varphi_0=P_0\B u.
	\end{split}
	\end{equation*}
	Therefore, on the one hand,
	\begin{equation*}
	\B P_0(\nu\partial_\omega W(\cdot,\omega_0)\varphi_0)=\bar\nu P_0\B(\partial_\omega W(\cdot,\omega_0)\varphi_0)=\bar\nu P_0\big(\partial_\omega\B W(\cdot,\omega_0)\B\varphi_0\big)=\bar\nu P_0(\partial_\omega W(\cdot,\omega_0)\varphi_0).
	\end{equation*}
	On the other hand,
	\begin{equation*}
		\begin{split}
			\B P_0(\nu\partial_\omega W(\cdot,\omega_0)\varphi_0)&=-\B P_0(g_{\omega_0}(\cdot)|\varphi_0|^\frac1\alpha\varphi_0)=-P_0\big(\B(g_{\omega_0}(\cdot))\B(|\varphi_0|^\frac1\alpha\varphi_0)\big)\\
			&=-P_0(g_{\omega_0}(\cdot)|\varphi_0|^\frac1\alpha\varphi_0).
		\end{split}
	\end{equation*}
	This yields $\nu\in\R$. Next, let us analyze $\phi$, i.e. the solution of \eqref{phi_GEN} in $Q_0D(A)$. Notice that if $P_0$ commutes with $\B$, then the same holds for $Q_0=Id-P_0$, too. Therefore, applying a similar argument, we get that $\B(\phi)$ satisfies \eqref{phi_GEN} too. Moreover, since $\B(\phi)\in D(A)\cap\langle\varphi_0^*\rangle^\perp$, we get $\B(\phi)=\phi$ because \eqref{phi_GEN} has a unique solution in $Q_0D(A)$. Let us now address $\sigma$ and $\psi$. We will show that the coupled fixed point problem \eqref{3_sys},\eqref{4_sys} preserves the realness of $\sigma$ and the $\PT$-symmetry of $\psi$. First, given $\sigma\in\R$ with $\sigma\in(-r_1,r_1)$, we prove that
	\begin{equation}\label{G_PTOF_PT}
	\B\psi=\psi\qquad\Rightarrow\qquad\B G(\sigma,\psi)=G(\sigma,\psi),
	\end{equation}
	where $G$ is defined in \eqref{4_sys_PTOF}. If \eqref{G_PTOF_PT} holds, then the fixed point $\psi=\psi(\sigma)$ of $\psi=G(\sigma,\psi)$ lies in $B_{r_2}(0)\cap\{\eta\in L^2(\R^d)\,|\,\B\eta=\eta\}$. To this aim, first we notice that for $u\in Q_0L^2(\R^d)$ it holds
	\begin{equation*}
	\begin{split}
	\B\big(Q_0L(\cdot,\omega_0)Q_0\big)^{-1}u&=\big(Q_0L(\cdot,\omega_0)Q_0\big)^{-1}\big(Q_0L(\cdot,\omega_0)Q_0\big)\B\big(Q_0L(\cdot,\omega_0)Q_0\big)^{-1}u\\
	&=\big(Q_0L(\cdot,\omega_0)Q_0\big)^{-1}\B\big(Q_0L(\cdot,\omega_0)Q_0\big)\big(Q_0L(\cdot,\omega_0)Q_0\big)^{-1}u\\
	&=\big(Q_0L(\cdot,\omega_0)Q_0\big)^{-1}\B u.
	\end{split}
	\end{equation*}
Next, recalling that $\varphi_0$, $\phi$, and $\psi$ are now $\PT$-symmetric, we get for $R$ (defined in \eqref{4_sys})
	\begin{equation*}
	\begin{split}
	\B R(\psi)&= Q_0\big(\B f(\cdot,\omega,\varphi)-\varepsilon^{\alpha+1}\B(g_{\omega_0}(\cdot)|\varphi_0|^\frac1\alpha\varphi_0)\big)+\varepsilon^{\alpha+1+\tau}\sigma Q_0(\partial_\omega\B(W(\cdot,\omega_0))\B(\varphi_0))\\
	&\quad+\varepsilon^{\alpha+2}Q_0\bigg(\frac{\nu^2}2\partial_\omega^2\B(W(\cdot,\omega_0))\B(\varphi_0)+\nu\partial_\omega\B(W(\cdot,\omega_0))\B(\phi)\bigg)+Q_0(\B v(\cdot,\omega,\varphi))\\
	&=R(\psi)+Q_0(\B v(\cdot,\omega,\varphi)-v(\cdot,\omega,\varphi)).
	\end{split}
	\end{equation*}
Inspecting all terms in $v$ appearing in \eqref{v} and exploiting $\sigma\in\R$ (and therefore $\omega\in\R$), we obtain that $\B v(\cdot,\omega,\varphi)=v(\cdot,\omega,\varphi)$, so \eqref{G_PTOF_PT} is proved. Finally, with similar manipulations one proves that $\B(\psi)=\psi$ implies $\B S(\sigma)=S(\sigma)$, where $S$ is defined in \eqref{3_sys_PTO}, obtaining thus $S(\sigma)\in\R$ since $\B S(\sigma)=\overline{S(\sigma)}$. Therefore, for a given $\PT-$symmetric $\psi(\sigma)$ the fixed point $\sigma$ of $S(\sigma)$ must be real.
	This completes the proof.
\end{proof}


\section{Applications to nonlinear surface plasmons}\label{Section_SPPs}

As mentioned in the introduction, we are interested in surface plasmon polaritons (SPPs) localized at one or more interfaces between different dielectric and metal layers. We consider the time harmonic and $z$-independent ansatz \eqref{E:ans-SPP} for the Maxwell system \eqref{Max}. A simple nonlinear, non-local relation for the displacement field is 
\begin{equation}\label{D}
\cD(x,y,z,t)=\cE+\int_\R\chi^{(1)}(x,y,z,t-s)\cE(s)\dd s +\int_\R \chi^{(3)}(x,y,z,t-s)\left((\cE\cdot\cE)\cE\right)(s)\dd s,
\end{equation}
where $\chi^{(1,3)}:\R^3\times \R\to \R$ and $\chi^{(1,3)}(\cdot,\tau)=0$ for $\tau<0$. The functions $\chi^{(1)}$ and $\chi^{(3)}$ are the linear and the cubic electric susceptibilities of the material, respectively. In general, $\chi^{(1)}$ and $\chi^{(3)}$ are tensors but in the isotropic case, which we assume, the relation in \eqref{D} with scalar $\chi^{(1)}$ and $\chi^{(3)}$ holds, see \cite[Section 2d]{MolNew}.

Substituting a monochromatic ansatz $\cE(x,y,z,t)=E(x,y,z)e^{\ri \omega t} + \text{c.c.}$ in \eqref{Max} and neglecting higher harmonics (terms proportional to $e^{3\ri \omega t}$ and $e^{-3\ri \omega t}$), we get $\cD(x,y,z,t)=D(x,y,z)e^{\ri\omega t}+\text{c.c.}$ with
$$D(x,y,z)=(1+\hat\chi^{(1)}(x,y,z,\omega))E+\hat\chi^{(3)}(x,y,z,\omega)(2|E|^2E+(E\cdot E)\bar E).$$
Here $|E|^2=E\cdot\overline{E}$ and $\hat{f}(\omega)$ is the Fourier-transform of $f$. Neglecting higher harmonics is a common approach in theoretical studies of weakly nonlinear optical waves \cite{shen1984}.
Using \eqref{Max}, we get that both $H$ and $D$ are curl-fields such that the divergence conditions $\nabla\cdot D=0, \nabla\cdot H=0$ hold automatically. Defining $c:=(\epsilon_0\mu_0)^{-1/2}$, in the second order formulation we have
$\frac{\omega^2}{c^2}D=\nabla\times\nabla\times E,$
i.e. 
\beq\label{E:curlcurl}
\nabla\times\nabla\times E - \frac{\omega^2}{c^2}(1+\hat\chi^{(1)}(x,y,z,\omega))E-\frac{\omega^2}{c^2}\hat\chi^{(3)}(x,y,z,\omega)(2|E|^2E+(E\cdot E)\bar E)=0.
\eeq

Recall again only odd nonlinearities are allowed  in $\cD$ when studying time harmonic waves. Even nonlinearities do not produce terms proportional to $e^{\ri\omega t}$.

We consider structures independent of $y$ and $z$, i.e. $\hat\chi^{(1,3)}=\hat\chi^{(1,3)}(x,\omega)$. Interfaces between layers are thus parallel  to the $yz$-plane.

For the TE-ansatz in \eqref{E:ans-SPP}, with only one nontrivial component, equation \eqref{E:curlcurl} reduces to the scalar problem
\begin{equation}\label{NLpb}
\varphi''+W(x,\omega)\varphi+\Gamma(x,\omega)|\varphi|^2\varphi=0,\quad x\in \R
\end{equation}
with
\begin{equation*}
W(x,\omega):=\frac{\omega^2}{c^2}(1+\hat\chi^{(1)}(x,\omega)) - k^2, \quad \Gamma(x,\omega):=\frac{3\omega^2}{c^2}\hat\chi^{(3)}(x,\omega).
\end{equation*} 

We study layers of $x$-periodic (including homogeneous) media. When a metal layer is homogeneous, we choose the simplest \textit{Drude model} of the linear susceptibility in that layer \cite{Pitarke_2006}
\begin{equation}\label{Drude}
\hat{\chi}^{(1)}(x,\omega)= -\frac{\omega_p^2}{\omega^2+\ri\gamma\omega},
\end{equation}
where $\omega_p\in\R^+$ and $\gamma\in\R$. For dielectric layers we choose a periodic (possibly constant) and generally complex $\hat\chi^{(1)}(\cdot,\omega)$. The imaginary part of $\hat\chi^{(1)}$ is related to the loss or gain of energy of an electromagnetic wave propagating inside the medium. For most of the materials, and in particular for metals, it is negative, corresponding to a \textit{lossy} material. However, in \textit{active} (doped) materials the energy of an electromagnetic wave is amplified (energy gain), and therefore the imaginary part of $\hat\chi^{(1)}$ is positive. Materials with a real $\hat\chi^{(1)}(\cdot,\omega)$ are called \textit{conservative}.

To make equation \eqref{NLpb} dimensionless, as well as in order to use the physical values of the parameters involved in the numerical study in Section \ref{S:3layers}, we introduce the new rescaled spatial variable, frequency, and wave number
	\begin{equation}\label{tildeRescaling}
	\tilde x:=\frac{\omega_p}c x,\quad\quad\tilde\omega:=\frac\omega{\omega_p},\quad\quad\tilde k:=\frac c{\omega_p} k,
	\end{equation}
	where $\omega_p$ is the bulk plasma frequency of a prescribed metal layer. Defining then $\tilde\varphi(\tilde x):=\varphi(x)$, we obtain the same equation as \eqref{NLpb} but in the tilde variables and with $W$ and $\Gamma$ respectively replaced by 
	\begin{equation}\label{tildeW}
	\tilde W(\tilde x,\tilde\omega):=\frac{c^2}{\omega_p^2}W\left(\frac c{\omega_p}\tilde x,\omega_p\tilde\omega\right)=\tilde\omega^2\left(1+\hat\chi^{(1)}\left(\frac c{\omega_p}\tilde x,\omega_p\tilde\omega\right)\right)-\tilde{k}^2,
	\end{equation}
	\begin{equation}\label{tildeGamma}
	\tilde\Gamma(\tilde x,\tilde\omega):=\frac{c^2}{\omega_p^2}\Gamma\left(\frac c {\omega_p}\tilde x,\omega_p\tilde\omega\right)=3\tilde\omega^2\hat\chi^{(3)}\left(\frac c {\omega_p}\tilde x,\omega_p\tilde\omega\right)\!.
	\end{equation}
	Note that for the Drude model the susceptibility is
$$
\hat\chi^{(1)}\left(\frac c{\omega_p}\tilde x,\omega_p\tilde\omega\right)= -\frac{1}{\tilde{\omega}^2+\ri\tilde\gamma\tilde\omega}\,,
$$
where $\tilde\gamma:=\gamma/\omega_p$.

For the sake of a simpler notation henceforth we will simply write $x,\omega,\gamma,k,W$, and $\Gamma$ instead of $\tilde x, \tilde\omega,\tilde\gamma,\tilde k, \tilde W$, and $\tilde \Gamma$.

Due to the presence of material interface(s), solutions of the Maxwell's equations \eqref{Max} are not smooth. However, they satisfy the interface conditions that the tangential component of $\cE$, the normal component of $\cD$ and the whole vector $\cH$ be continuous across each interface, see Sec. 33-3 in \cite{feynman2}. For our interfaces parallel to the $yz$-plane we get \eqref{E:IFC}. For the ansatz \eqref{E:ans-SPP} the interface conditions reduce to a $C^1$-continuity condition on $\varphi$
\begin{equation}\label{IFCs}\tag{IFCs}
\llbracket \varphi\rrbracket =\llbracket \partial_x\varphi \rrbracket =0.
\end{equation}
Our nonlinear problem \eqref{eq} is thus equation \eqref{NLpb} with $W$ and $\Gamma$ respectively replaced by \eqref{tildeW}-\eqref{tildeGamma} and coupled with the \eqref{IFCs}.

To apply our bifurcation result to \eqref{NLpb} with \eqref{IFCs}, a real, linear eigenvalue is needed. We show that such an eigenvalue exists in some $\PT$-symmetric choices of the layers. SPPs in $\PT$-symmetric structures have been studied in the physics literature before by employing active materials, see e.g. \cite{AD14,Barton2018,Han_2014,Wang_2017}. Nevertheless, we are not aware of a rigorous mathematical existence proof of the asymptotic expansion of nonlinear SPPs in the frequency dependent case.

Theorem \ref{Theorem_BIF_NLeigv} can be applied to \eqref{NLpb} with finitely many interfaces (at $x=x_1,\dots,x_N$) using the following natural choice of $A$, $D(A)$, and $f$:
$$A:=-\frac{d^2}{dx^2},\quad\;\; f(x,\omega,\varphi):=\Gamma(x,\omega)|\varphi|^2\varphi,$$
and
\begin{equation}\label{E:DA}
\begin{split}
D(A):=\big\{&\varphi\in L^2(\R)\;:\;\varphi|_{(x_j,x_{j+1})}\in H^2((x_j,x_{j+1}),\C)\;\mbox{for}\;j=0,\dots N\\
&\mbox{and}\;\, \llbracket \varphi\rrbracket_j =\llbracket \partial_x\varphi \rrbracket_j=0,\, j=1,\dots,N\big\},
\end{split}
\end{equation}
where $\llbracket\varphi\rrbracket_j:=\lim_{x\to x_j+}\varphi(x)-\lim_{x\to x_j-}\varphi(x)$ and $x_0:=-\infty$, $x_{N+1}:=\infty$. It is easy to see that $D(A)=H^2(\R,\C)$ using the definition of the second weak derivative and the fact that at the interfaces any $\varphi\in D(A)$ is of class $C^1$. Note that assumptions (A1)-(A2), (f1)-(f4) are satisfied for any $\omega \in \C\setminus \{0,-\ri \gamma\}$.
\subsection{The linear eigenvalue problem}
In order to apply Theorem \ref{Theorem_BIF_NLeigv}, we need to find an eigen-pair $(\omega_0,\varphi_0)$ of the \textit{linear} problem
\begin{equation}\label{Lpb}
\varphi''+W(x,\omega)\varphi=0,\quad x\in\R
\end{equation}
coupled with \ref{IFCs}, with a simple and isolated $\omega_0\in\C\setminus\{0\}$. 
To ensure the realness of the frequency we need, in fact, $\omega_0\in\R\setminus\{0\}$ as well as the $\PT$-symmetry of $W(\cdot,\omega)$, such that Proposition \ref{Thm_PT} can be applied. 
We shall see that the existence of a simple and isolated eigenvalue $\omega_0\in\R\setminus\{0\}$ strongly depends on the choice of the layers. As we show in Sec. \ref{S:2layers}, the choice of two layers ($N=1$) of periodic materials with one being conservative and the other a non-conservative homogeneous material (i.e. with a complex $\hat{\chi}^{(1)}$) leads to no real eigenvalues $\omega_0$. On the other hand, in Sec. \ref{S:3layers} we find two $\PT$-symmetric settings with three homogeneous layers ($N=2$) leading to the existence of an isolated simple eigenvalue $\omega_0\in \R$ in \eqref{Lpb}. These settings are: (active dielectric - conservative Drude metal - lossy dielectric) and a hypothetical setting of (Drude metal with gain - lossless dielectric - lossy Drude metal). 

\subsubsection{Two periodic layers}\label{S:2layers}
We consider first the case of two layers, each being either a periodic metal or a periodic dielectric, where we set the interface at $x=0$. Hence 
$$W(x,\omega) = W_\pm(x,\omega) \quad \text{for } \pm x>0,$$
where the functions $W_\pm(x,\omega)=\omega^2(1+\hat{\chi}^{(1)}_\pm(x,\omega))-k^2$ are periodic in $x$ with periods $\nu_\pm >0$. The governing linear problem \eqref{Lpb} has two linearly independent Bloch wave solutions $\psi^\pm_{1,2}$ on the half line $\pm x>0$ respectively. The Bloch wave theory for the Hill's equation \eqref{Lpb} can be found in \cite{Eastham}. A necessary condition for the existence of an $L^2(\R,\C)$-solution of \eqref{Lpb} is
$$0\notin S:=\sigma\left(-\frac{d^2}{dx^2}- W_+(\cdot,\omega)\right) \cup \sigma\left(-\frac{d^2}{dx^2}- W_-(\cdot,\omega)\right).$$
Otherwise (if $0\in S)$, on at least one of the half lines there is no decaying solution. If $0\notin S$, the solutions $\psi^\pm_{1,2}$ have the form
$$
\psi^+_{j}(x)=p^+_j(x)e^{(-1)^j\lambda_+ x}, \qquad \psi^-_{j}(x)=p^-_j(x)e^{(-1)^j\lambda_- x},\quad j=1,2,
$$
where $\Real(\lambda_\pm) >0, p^\pm_j(x+\nu_\pm)=p^\pm_j(x)$ for all $x\in \R$ and both $j=1,2$. If, say, $W_+$ is real, then $p^+_j$ and $\lambda_+$ can be chosen real but $p_j^+$ is generally $2\nu_+$-periodic, see \cite{Eastham}. 

An $L^2(\R)$-solution is given by
$$
\varphi(x)=\begin{cases}\varphi^+(x):=p_1^+(x)e^{-\lambda_+ x}, & x>0,\\
\varphi_-(x):=p_2^-(x)e^{\lambda_- x}, & x<0.
\end{cases}
$$
Due to the linearity of \eqref{Lpb} the $C^1$-matching condition \eqref{IFCs} is equivalent to the condition
\beq\label{E:R-cond}
R_+:=\frac{\varphi_+'(0)}{\varphi_+(0)}=\frac{\varphi_-'(0)}{\varphi_-(0)}=:R_-.
\eeq
Note that by varying the parameters $\omega,k\in \R$ we get $R_\pm=R_\pm(\omega,k)$. 

In \cite{DPR09} the case of real, periodic $W_\pm$, i.e. the one of two periodic conservative materials, was considered and eigenvalues were found by varying $k$ and searching for zeros of $R_+(k)-R_-(k)$.

In the presence of non-conservative materials the respective potential $W$ is complex. It is easy to see that the single interface of a conservative material (e.g. a classical dielectric) with a real $W$ and non-conservative homogeneous material (e.g. a Drude metal) with a complex $W$ does not support any eigenvalues of \eqref{Lpb}. Note that this does not contradict the existence of SPPs at such single metal/dielectric interfaces in general because this existence holds for TM-polarisations, see \cite{Pitarke_2006}. Without loss of generality we assume that the conservative material is on the half line $x>0$, i.e. $W_+(x,\omega)$ is real. The non-conservative material in $x<0$ is homogeneous (described, e.g. by \eqref{Drude}) such that $W_-(x,\omega)=W_-(\omega)$ is complex and independent of $x$. Hence $\varphi_-(x)=ce^{\lambda_- x}$ with\footnote{Henceforth, for any $z\in\C$ we choose the square root $\sqrt{z}$ as the one solution $a$ of $a^2=z$ with $\text{arg}(a)\in (-\pi/2,\pi/2]$.} $\lambda_-=\sqrt{W_-(\omega)}\in \C\setminus \R$, such that 
$$R_-=\lambda_-,$$
and $\varphi_+(x)=p_1^+(x)e^{-\lambda_+ x}$ with $\lambda_+\in \R$ and $p_1^+(x)$ real and $2\nu_+-$periodic, such that
$$R_+=\frac{p_1^{+'}(0)}{p_1^+(0)}-\lambda_+\in \R.$$
Note that $p_1^+(0)=0$ implies $c=0$ due to \eqref{IFCs} such that only the trivial solution $\varphi\equiv 0$ is produced in that case. 

Because $R_-\in \C\setminus \R$ and $R_+\in\R$, condition \eqref{E:R-cond} is not satisfied and no eigenvalue exists.

\subsubsection{Three homogeneous layers}\label{S:3layers}

Next we consider three homogeneous material layers with interfaces at $x=0$ and $x=d>0$, i.e. a sandwich geometry with two unbounded layers, 
\begin{equation}\label{chi1_w_3layers_homog}
\hat{\chi}^{(1)}(x,\omega):=\begin{cases}
\eta_- & \text{for $x<0$,} \\
\eta_* & \text{for $x\in(0,d)$,} \\
\eta_+ & \text{for $x>d$,}
\end{cases}
\quad\mbox{i.e.}\quad
W(x,\omega)=\begin{cases}
\omega^2(1+\eta_-)-k^2=:W_- & \text{for $x<0$,} \\
\omega^2(1+\eta_*)-k^2=:W_* & \text{for $x\in(0,d)$,} \\
\omega^2(1+\eta_+)-k^2=:W_+ & \text{for $x>d$,}
\end{cases}
\end{equation}
where $\eta_\pm,\eta_*\in\C$. A localized solution of \eqref{Lpb} is possible only if 
\beq\label{E:0-spec}
0\notin S:=\bigcup_{\pm} \sigma\left(-\frac{d^2}{dx^2}-W_\pm(\omega)\right).
\eeq
Note that if the semi-infinite layers are conservative, then $\eta_+,\eta_-\in\R$ and \eqref{E:0-spec} is equivalent to 
$$k^2>\omega^2(1 + \max\{\eta_+,\eta_-\}).$$
Under assumption \eqref{E:0-spec} we have 
\begin{equation}\label{eigenfct}
\varphi(x)=\begin{cases}
Ae^{\lambda_-x} & \text{for $x<0$,} \\
Be^{\mu x}+Ce^{-\mu x} & \text{for $x\in(0,d)$,} \\
De^{-\lambda_+x} & \text{for $x>d$,}
\end{cases}
\end{equation}
where $\mu:=\sqrt{-W_*}$, $\lambda_\pm:=\sqrt{- W_\pm}$, $\Real(\lambda_\pm)>0$, and where $A,B,C,D\in \C$ are constants to be determined. We can normalize such that $D=1$. Then the $C^1$-matching at $x=0$ and $x=d$ is equivalent to
\begin{equation}\label{parametersBCD}
\begin{cases}
A=B+C\\
A\lambda_-=\mu B-\mu C
\end{cases}
\quad
\begin{cases}
e^{-\lambda_+d}=Be^{\mu d}+Ce^{-\mu d}\\
-\lambda_+e^{-\lambda_+d}=\mu Be^{\mu d}-\mu Ce^{-\mu d}.
\end{cases}
\end{equation}
This system has the unique solution
$$A=\frac{e^{-\lambda_+d}}2\left[\left(1-\frac{\lambda_+}{\mu}\right)e^{-\mu d}+\left(1+\frac{\lambda_+}{\mu}\right)e^{\mu d}\right],$$
$$B=\frac12\left(1-\frac{\lambda_+}{\mu}\right)e^{-(\mu+\lambda_+)d},\quad C=\frac12\left(1+\frac{\lambda_+}{\mu}\right)e^{-(-\mu+\lambda_+)d},$$ 
together with a condition on the parameter $d$:
\begin{equation}\label{cond_d_cD|M|cD}
\exists m\in \Z: \  d=\tilde d_m:=\frac1{2\mu}\left[\log\left(\frac{(\mu-\lambda_-)(\mu-\lambda_+)}{(\mu+\lambda_-)(\mu+\lambda_+)}\right)+2\pi \ri m\right]\in (0,\infty).
\end{equation}
The term $2\pi \ri m$ appears due the fact that $z=\log(b)+2\pi \ri m$ solves $e^z=b$ for any $m\in \Z$. Condition \eqref{cond_d_cD|M|cD} means that under assumption \eqref{E:0-spec} there is a width $d>0$ supporting a real eigenvalue if and only if $\tilde d_m=\tilde d_m(\omega)$ is positive for some value of $\omega \in\R$ and some $m\in \Z$. If all the layers are homogeneous conservative materials, \eqref{cond_d_cD|M|cD} cannot be satisfied because $\mu>0$ and the argument of the logarithm on the right hand side of \eqref{cond_d_cD|M|cD} lies in $(0,1)$ such that $\mbox{Re}(\tilde d_m)<0$ for all $m\in \Z$.

Next, we consider three sandwich settings, out of which the last two are $\PT$-symmetric. As the numerical evaluation of $\tilde{d}_m$ suggests, both of these apparently lead to the existence of linear eigenvalues $\omega_0\in \R$ and hence to real bifurcating nonlinear eigenvalues $\omega$. The first one of these $\PT$-symmetric cases has been taken from the physics literature \cite{AD14,Barton2018} while the second one corresponds to a hypothetical material. 

\paragraph{Case 1: conservative dielectric - Drude metal - conservative dielectric.} 

Note that a Drude metal layer being sandwiched between two conservative dielectrics does not produce a $\PT$-symmetric potential $W$. Hence, we do not expect a real eigenvalue $\omega$.

We have $\eta_\pm>0$ and $\eta_*=-\frac{1}{\omega^2+\ri\gamma \omega}\in \C$. Our numerical study of $\tilde d_m$ shows that for any choice of the constants $\eta_+,\eta_-,\gamma\in\R$ and $m\in \Z$ we cannot find a frequency $\omega$ for which $\tilde d_m(\omega)\in (0,\infty)$.
Figure \ref{MDM_plot} (a), (b) shows as an example the behaviour of the maps $\omega\mapsto\Real(\tilde d_0(\omega))$ and $\omega\mapsto\Imag(\tilde d_0(\omega))$ for different choices of $\gamma,\eta_+,\eta_-$.
\begin{figure}[ht!]
	\centering
	\begin{subfigure}[b]{0.4\linewidth}
		\includegraphics[width=\linewidth]{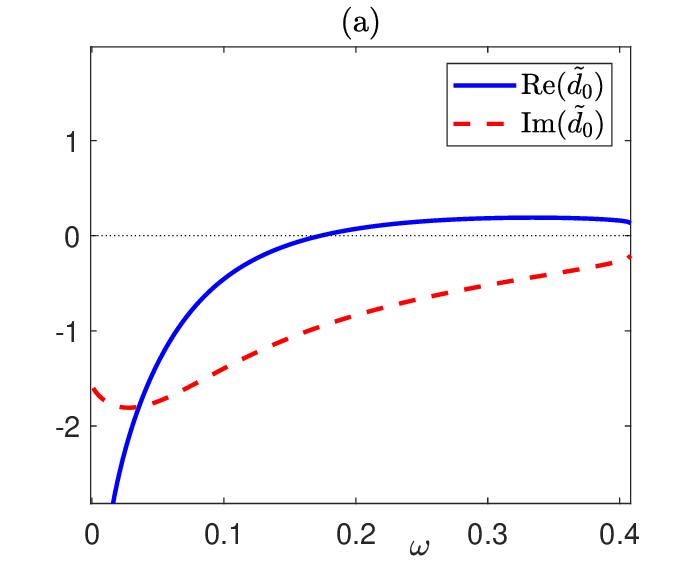}
	\end{subfigure}
	\quad
	\begin{subfigure}[b]{0.4\linewidth}
		\includegraphics[width=\linewidth]{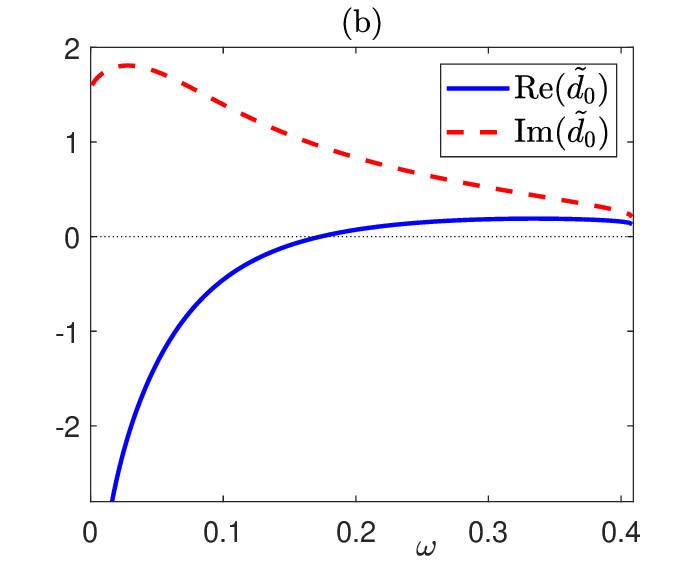}
	\end{subfigure}
	\begin{subfigure}[b]{0.4\linewidth}
		\includegraphics[width=\linewidth]{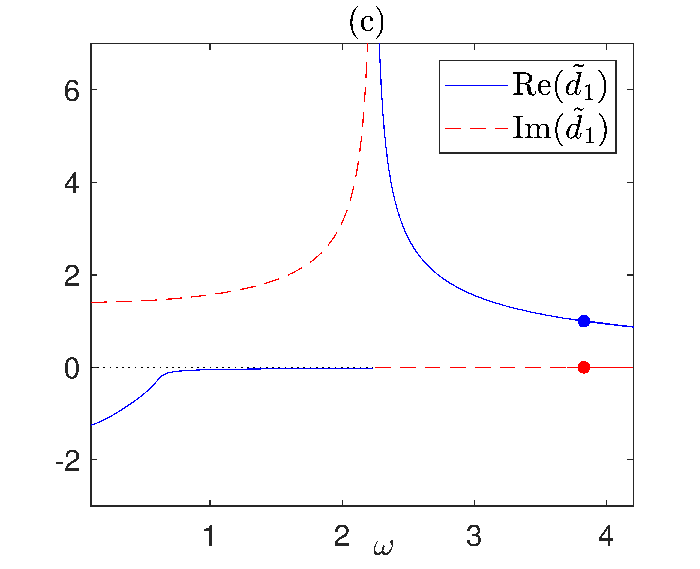}
	\end{subfigure}
	\quad
	\begin{subfigure}[b]{0.4\linewidth}
		\includegraphics[width=\linewidth]{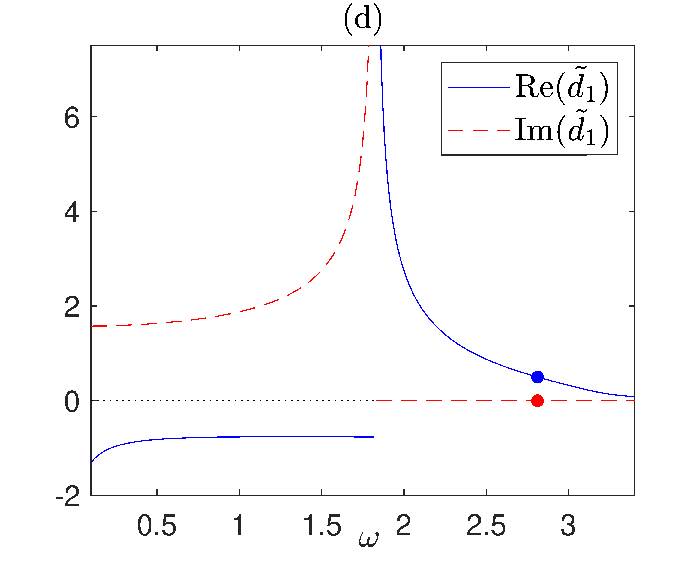}
	\end{subfigure}
	\caption{{\footnotesize The graph of $\omega\mapsto \Real(\tilde d_m(\omega))$ (full blue line) and $\omega\mapsto \Imag(\tilde d_m(\omega))$ (dashed red line).  
	(a), (b) Case 1 with $k=1$, $\gamma=\mp0.5$, $\eta_+=5$, $\eta_-=0.05$, and $m=0$. (c) Case 2 with $k=2$, $\eta_\pm=9.2\pm1.28\ri$, $\gamma=0$ and $m=1$. (d) Case 3 with $k=2$, $\gamma^+=-\gamma^-=0.5$, $\eta_*=0.2$ and $m=1$. In (c) (resp. (d)) the chosen value $d=1$, attained at $\omega\approx 3.8275$ (resp. $d=0.5$, attained at $\omega\approx 2.8096$), is highlighted in the graph.}}
	\label{MDM_plot}
\end{figure}

\paragraph{Case 2: $\PT$-symmetric (dielectric - metal - dielectric) setting.}

For the case of a homogeneous Drude metal sandwiched between two homogeneous non-conservative dielectric layers, we have 
\beq\label{E:eta_gD0MlD}
\eta_\pm=\eta_R^\pm+\ri \eta_I^\pm\qquad\text{and}\qquad \eta_* = -\frac{1}{\omega^2+\ri\gamma\omega}.
\eeq
This setting leads to a $\PT$-symmetric $W$ (with respect to $x=d/2$, i.e. $W(\frac d2-x,\omega)=\overline{W(\frac d2+x,\omega)}$ for all $x\in\R$ and $\omega\in\R\setminus\{0\}$) if we choose
$$\eta_R^+=\eta_R^-,\quad \eta_I^+=-\eta_I^-,\quad\gamma=0.
$$
Hence, one of the dielectric layers is lossy while the other is active and generates energy gain.

Note that in this example the simple transformation $\omega':=\omega^2$ leads to a linear dependence on the spectral parameter $\omega'$.

This configuration of a conservative metal sandwiched between a couple of well-prepared active and lossy dielectrics was considered, e.g., in \cite{AD14,Barton2018}. As the active material we consider titanium dioxide (TiO$_2$), with refractive index $n=3.2+0.2\ri$, and as the metal we choose silver with the bulk plasma frequency $\omega_p(\text{Ag})=8,85\cdot 10^{15}\text s^{-1}$. We use $\omega_p(\text{Ag})$ as the rescaling parameter in \eqref{tildeRescaling}. Recalling the relation $W(x)=\omega^2n(x)^2-k^2$ between the refractive index and the potential $W$, we choose $k=2$ and obtain $\eta_\pm=9.2\pm1.28\ri$.

Our numerical tests show that this $\PT$-symmetric setting leads to $\tilde d(\omega)>0$ for all $\omega>\omega_{DMD}\approx 2.23$, see  Figure \ref{MDM_plot}(c). For the computation of the bifurcation we choose the point $\omega_0\approx 3.8275$, for which $d=\tilde{d}_{1}(3.8275) \approx 1$.

As the graph in  Figure \ref{MDM_plot}(c) suggests, $\frac d{d\omega}\Real(\tilde d_{1}(\omega))\neq 0$ for all $\omega>\omega_{DMD}$. Hence, for any width $d\in\tilde d_{1}((\omega_{DMD},\infty))$ there is only one $\omega_0$ for which $d=\tilde d_{1}(\omega_0)$, i.e. for which \eqref{cond_d_cD|M|cD} holds. We denote the corresponding linear eigenfunction $\varphi(\cdot,\omega_0)$ by 
$$\widetilde\varphi_0:=\varphi(\cdot,\omega_0),$$
see \eqref{eigenfct}. Clearly, $\varphi_0=\widetilde\varphi_0/\|\widetilde\varphi_0\|$.

\paragraph{Case 3: $\PT$-symmetric (metal - dielectric - metal) setting.}

For the case of a homogeneous conservative dielectric sandwiched between two homogeneous Drude metal layers we have 
\beq\label{E:eta_MDM}
\eta_\pm = -\frac{2\pi}{\omega^2+\ri\gamma^\pm \omega}\in \C \qquad \text{ and } \quad  \eta_*>0.
\eeq
This setting leads to a $\PT$-symmetric $W$ with respect to $x=d/2$ if we choose $\omega\in\R\setminus\{0\}$ and $\gamma^+=-\gamma^-.$
Note that this setting (with $\gamma_+\neq 0$) is hypothetical as materials with $\hat\chi^{(1)}(\omega) = - 1/(\omega^2+\ri \gamma \omega)$ and a negative $\gamma$ may not exist.

Also notice that here, unlike the previous example, the dependence of $W$ on $\omega$ is truly nonlinear. We retrieve numerically again a similar plot for the function $\omega\mapsto\tilde d(\omega)$, as shown in Figure \ref{MDM_plot}(d): $\tilde d(\omega)>0$ for all $\omega>\omega_{MDM}$ with some $\omega_{MDM}>0$. For $k=2,\gamma^+=0.5$ and $\eta_*=0.2$ we get $\omega_{MDM}\approx 1.83$. For the computation of the bifurcation we choose the point $\omega_0\approx 2.8096$ for which $d=\tilde d_1(\omega_0)\approx 0.5$. 

\subsection{Bifurcation of a nonlinear eigenvalue}
We aim to apply the bifurcation result of Theorem \ref{Theorem_BIF_NLeigv} in its symmetric version provided by Proposition \ref{Thm_PT} in both settings given by Case 2 and Case 3 and find a bifurcating branch of solutions to the reduced Maxwell's equation \eqref{NLpb}. In both cases we choose the cubic susceptibility $\hat\chi^{(3)}\equiv 1$. We need thus to verify that $\omega_0$, given by the numerical discussion in Section \ref{S:3layers} for a fixed suitable layer width $d$, is a simple isolated eigenvalue in the sense of (E1)-(E2). 

\paragraph{Verification of (E1): $\boldsymbol{\omega_0}$ is simple}

Because in \eqref{eigenfct}-\eqref{parametersBCD}  the constants $A,B,C,D$ are unique (up to normalization of $\widetilde\varphi_0$), it is clear that $\ker L(\cdot,\omega_0)=\spann\,\widetilde\varphi_0$, so $0$ as eigenvalue of $L(\cdot,\omega_0)$ is geometrically simple. To prove the algebraic simplicity, suppose by contradiction that there exists a Jordan chain associated to $\omega_0$. This means, there exists $u\in D(A)$ (see \eqref{E:DA}) such that
\begin{equation}\label{Lpb_Jordan}
	L(\cdot,\omega_0)u=-u''-W(\cdot,\omega_0)u=\widetilde\varphi_0\qquad\mbox{in }\,\R.
\end{equation}
Solving \eqref{Lpb_Jordan} explicitly using the variation of constants, one finds
\begin{equation}\label{u}
	u(x)=\begin{cases}
		c_1e^{\lambda_-x}+\frac1{2\lambda_-}\left(e^{\lambda_-x}\int_x^0e^{-\lambda_-t}\widetilde\varphi_0(t)\dd t+e^{-\lambda_-x}\int_{-\infty}^xe^{\lambda_-t}\widetilde\varphi_0(t)\dd t\right) & \text{for $x<0$,} \\
		c_2e^{\mu x}+c_3e^{-\mu x}+\frac1{2\mu}\left(e^{\mu x}\int_x^de^{-\mu t}\widetilde\varphi_0(t)\dd t+e^{-\mu x}\int_0^xe^{\mu t}\widetilde\varphi_0(t)\dd t\right) & \text{for $x\in(0,d)$,} \\
		c_4e^{-\lambda_+x}+\frac1{2\lambda_+}\left(e^{\lambda_+x}\int_x^{+\infty}\!e^{-\lambda_+t}\widetilde\varphi_0(t)\dd t+e^{-\lambda_+x}\int_d^xe^{\lambda_+t}\widetilde\varphi_0(t)\dd t\right) & \text{for $x>d$.}
	\end{cases}
\end{equation}
In order to belong to $D(A)$, $u$ must satisfy the $C^1$-matching at the interfaces $x=0$ and $x=d$. This implies that the constants $c_1$, $c_2$, $c_3$, $c_4$ have to solve the linear system
$$T\left(c_1,c_2,c_3,c_4\right)^\trans=b,$$
where
$$T:=\left(\begin{matrix}
	-1 & 1 & 1 & 0 \\
	-\lambda_- & \mu & -\mu & 0 \\
	0 & e^{\mu d} & e^{-\mu d} & -e^{-\lambda_+d} \\
	0 & \mu e^{\mu d} & -\mu e^{-\mu d} & \lambda_+e^{-\lambda_+d}
\end{matrix}\right)$$
and
$$2b:=\left(\begin{matrix}
	\frac1{\lambda_-}\int_{-\infty}^0e^{\lambda_-t}\widetilde\varphi_0(t)\dd t-\frac1\mu\int_0^de^{-\mu t}\widetilde\varphi_0(t)\dd t\\
	-\int_{-\infty}^0e^{\lambda_-t}\widetilde\varphi_0(t)\dd t-\int_0^de^{-\mu t}\widetilde\varphi_0(t)\dd t\\
	- \frac{e^{-\mu d}}\mu\int_0^de^{\mu t}\widetilde\varphi_0(t)\dd t+\frac{e^{\lambda_+ d}}{\lambda_+}\int_d^{+\infty}\!e^{-\lambda_+t}\widetilde\varphi_0(t)\dd t\\
	e^{-\mu d}\int_0^de^{\mu t}\widetilde\varphi_0(t)\dd t+e^{\lambda_+ d}\int_d^{+\infty}\!e^{-\lambda_+t}\widetilde\varphi_0(t)\dd t
\end{matrix}\right)\,.$$
Note that $T$ is singular since
\begin{equation*}
	\det T = e^{-(\lambda_++\mu)d}\left((\lambda_+-\mu)(\lambda_--\mu)- e^{2\mu d}(\lambda_++\mu)(\lambda_-+\mu)\right)=0
\end{equation*}
by our choice of $d$ in \eqref{cond_d_cD|M|cD}. In order to find a contradiction and exclude the existence of a solution $u\in D(A)$ of \eqref{Lpb_Jordan}, we now prove that $b$ is \textit{not} orthogonal to the kernel of $\overline T^\trans$. Standard computations show that $\ker\overline T^\trans$ is one-dimensional and given by
\begin{equation*}
	\ker\overline T^\trans=\spann\,p:=\spann\,\left(\overline{\lambda_-}\left(\overline\mu-\overline{\lambda_+}\right)e^{-\overline\mu d}, \left(\overline\mu-\overline{\lambda_+}\right)e^{-\overline\mu d}, \overline{\lambda_+}\left(\overline{\lambda_-}+\overline\mu\right), \overline{\lambda_-}+\overline\mu\right)^\trans.
\end{equation*}
The scalar product $\left(b,p\right)$ then reads
\begin{equation*}
	\begin{split}
		\left(b,p\right)&=-2(\lambda_+-\mu)e^{-\mu d}\int_{-\infty}^0e^{\lambda_-t}\widetilde\varphi_0(t)\dd t+2e^{\lambda_+ d}(\lambda_-+\mu)\int_d^\infty e^{-\lambda_+t}\widetilde\varphi_0(t)\dd t\\
		&\quad+\frac{e^{-\mu d}(\lambda_+-\mu)}{\mu}\left[(\lambda_--\mu)\int_0^de^{-\mu t}\widetilde\varphi_0(t)\dd t-(\lambda_-+\mu)\int_0^de^{\mu t}\widetilde\varphi_0(t)\dd t\right]\,.
	\end{split}
\end{equation*}
After evaluating the integrals for $\widetilde\varphi_0$ given by \eqref{eigenfct}-\eqref{parametersBCD}, and some algebraic computations in which the identity $e^{2\mu d}=\frac{(\mu-\lambda_+)(\mu-\lambda_-)}{(\mu+\lambda_+)(\mu+\lambda_-)}$ is frequently used, we get
\begin{equation*}
	\begin{split}
		\left(b,p\right)=\frac{e^{-\lambda_+ d}}{2(\lambda_--\mu)}\bigg[\frac{\lambda_-^2-\mu^2}{\lambda_+}+\frac{\lambda_+^2-\mu^2}{\lambda_-}-\frac{ d(\lambda_+^2-\mu^2)(\lambda_-^2-\mu^2)+(\lambda_-+\lambda_+)(\lambda_-\lambda_+-\mu^2)}{\mu^2}\bigg]\,.
	\end{split}
\end{equation*}

Next, we check that $\left(b,p\right)$ is non-zero for the values of $\lambda_\pm$, $\mu$, and $d$ obtained numerically in Section \ref{S:3layers}. We obtain $\left(b,p\right)\approx -19.38 -46.36 \ri$ for Case 2 and  $\left(b,p\right)\approx 0.82 + 1.58 \ri$ for Case 3.

\paragraph{Verification of (E2): $\boldsymbol{\omega_0}$ is isolated} 
Because for $L(x,\omega_0):=-\frac{d^2}{dx^2}-W(x,\omega_0)$ the essential spectrum is given by 
$$\sigma_\text{ess}\left(L(\cdot,\omega_0)\right)=\bigcup_\pm \left\{\lambda-W_\pm(\omega_0): \lambda\in [0,\infty)\right\}\,.$$
Since $W_\pm(\omega_0)\in\C\setminus\R$ in both Case 2 and Case 3, we have $0\notin \sigma_\text{ess}\left(L(\cdot,\omega_0)\right)$. Since the essential spectrum is closed, $0$ is isolated from $\sigma_\text{ess}\left(L(\cdot,\omega_0)\right)$.

Next, we show the isolatedness of $0$ from other eigenvalues of $L(\cdot,\omega_0)$.  If $\kappa\in\C\setminus\{0\}$ is an eigenvalue of $L(\cdot,\omega_0)$, it must be
\begin{equation*}
	\begin{aligned}
	&\frac1{\mu(0)}\left[\log\left(\frac{(\mu(0)-\lambda_-(0))(\mu(0)-\lambda_+(0))}{(\mu(0)+\lambda_-(0))(\mu(0)+\lambda_+(0))}\right)\!+2\pi\ri m\right]\\
	&=\frac1{\mu(\kappa)}\!\left[\log\left(\frac{(\mu(\kappa)-\lambda_-(\kappa))(\mu(\kappa)-\lambda_+(\kappa))}{(\mu(\kappa)+\lambda_-(\kappa))(\mu(\kappa)+\lambda_+(\kappa))}\right)\!+2\pi\ri\tilde m\right]
		\end{aligned}
\end{equation*}
for some $\tilde m\in\Z$, where $\mu(\kappa):=\sqrt{-W_*(\omega_0)-\kappa}$ and $\lambda_\pm(\kappa):=\sqrt{-W_\pm(\omega_0)-\kappa}$.

Suppose there exists a sequence $(\kappa_j)_j\subset\C$ of such eigenvalues of $L(\cdot,\omega_0)$ which converges to $0$. Then, by continuity of the maps $\kappa\mapsto\mu(\kappa)$ and $\kappa\mapsto\lambda_\pm(\kappa)$, we infer $\tilde m=m$ and therefore
\begin{equation}
	f(\kappa_j):=\mu(\kappa_j)\left(\log g(0)+2\pi\ri m\right)-\mu(0)\left(\log g(\kappa_j)+2\pi\ri m\right)=0
\end{equation}
must hold, where
\begin{equation*}
	g(\kappa):=\frac{(\mu(\kappa)-\lambda_-(\kappa))(\mu(\kappa)-\lambda_+(\kappa))}{(\mu(\kappa)+\lambda_-(\kappa))(\mu(\kappa)+\lambda_+(\kappa))}.
\end{equation*}
Since $\kappa\mapsto f(\kappa)$ is differentiable at $0$, a necessary condition for $f(\kappa_j)=0$ with $\kappa_j\to0$ is $f'(0)=0$, that is
\begin{equation}\label{f_der}
	f'(0)=\mu'(0)\left(\log g(0)+2\pi\ri m\right)-\mu(0)\frac{g'(0)}{g(0)}=0.
\end{equation}
By simple computations, one gets
\begin{equation*}
	g'(0)=\frac{g(0)}{\mu(0)}\left(\frac1{\lambda_+(0)}+\frac1{\lambda_-(0)}\right),
\end{equation*}
from which, by \eqref{f_der} and the definition of $d$ in \eqref{cond_d_cD|M|cD}, one infers
$$-d-\left(\frac1{\lambda_+(0)}+\frac1{\lambda_-(0)}\right)=0.$$
This contradicts the fact that $d>0$ and that $\lambda_\pm$ are chosen with positive real part. As a consequence, we deduce that $\omega_0$ is isolated in the sense of (E2).

\paragraph{Bifurcation diagrams}

The numerical computations below were obtained using a centered finite difference discretization of fourth order on an equispaced grid and with the condition that $\varphi(x)=0$ for $x$ outside the computational interval. The linear eigenfunction $\varphi_0$ was computed using Matlab's built in eigenvalue solver ${\rm eigs}$. The computation of the nonlinear solution $\varphi$ was done via the standard Newton's iteration.

\paragraph{Case 2} In this case the linear susceptibility $\hat\chi^{(1)}$ is chosen as in \eqref{chi1_w_3layers_homog}, \eqref{E:eta_gD0MlD} with parameters as in Section \ref{S:3layers} ($k=2$, $\gamma=0$ and $\eta_\pm=9.2\pm1.28\ri$). The linear eigenvalue is selected as $\omega_0=3.8275$. The last condition to check is (Wt). A numerical approximation produces $\langle\partial_\omega W(\cdot,\omega_0)\varphi_0,\varphi_0^*\rangle \approx 7.602.$
\begin{figure}[h!]
	\centering
	\begin{subfigure}[b]{0.4\linewidth}
		\includegraphics[width=\linewidth]{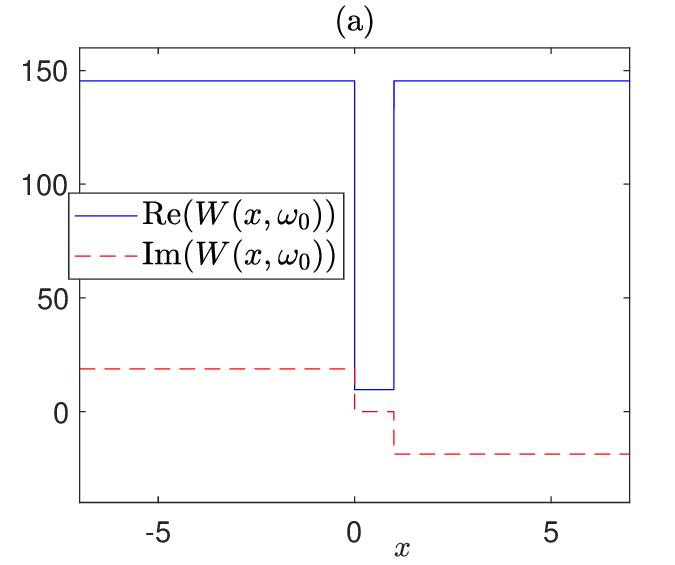}
	\end{subfigure}
	\quad
	\begin{subfigure}[b]{0.4\linewidth}
		\includegraphics[width=\linewidth]{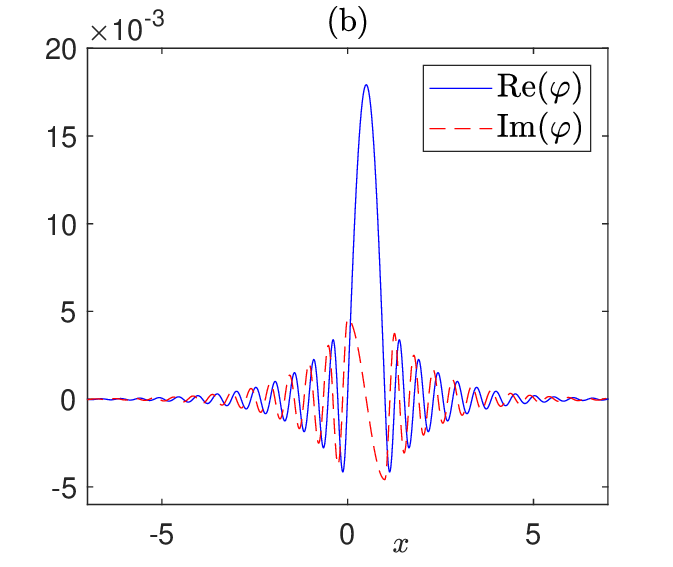}
	\end{subfigure}	
	\begin{subfigure}[b]{0.4\linewidth}
		\includegraphics[width=\linewidth]{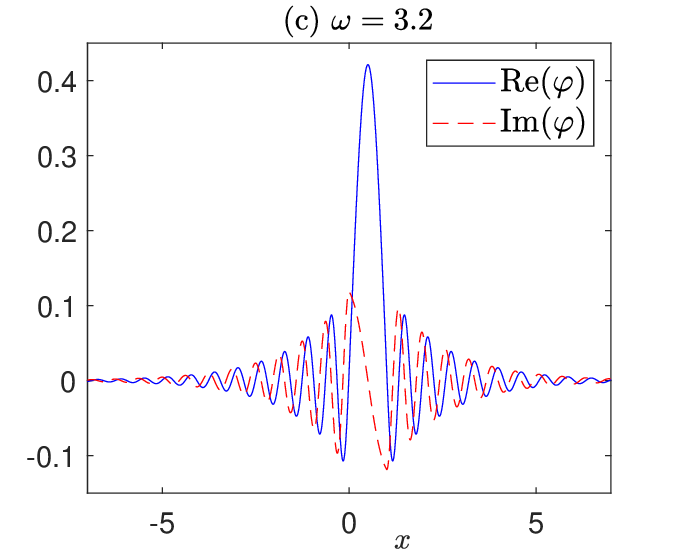}
	\end{subfigure}
	\quad
	\begin{subfigure}[b]{0.4\linewidth}
		\includegraphics[width=\linewidth]{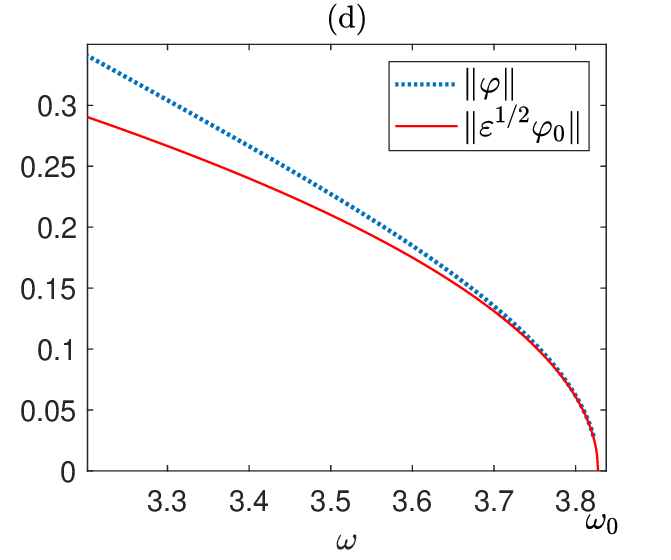}
	\end{subfigure}
	\caption{{\footnotesize Equation \eqref{NLpb}, \eqref{chi1_w_3layers_homog}, \eqref{E:eta_gD0MlD} with the parameters: $d=1$, $\eta_\pm=9.2\pm1.28\ri$, and $k=2$. (a) The potential $W(\cdot,\omega_0)$ at $\omega_0\approx 3.8275$. (b) The eigenfunction $\varphi_0$ of  \eqref{Lpb} at $\omega_0$. (c) The solution of \eqref{NLpb} at $\omega=3.2$. (d) The bifurcation diagram $(\omega, \|\varphi\|)$ (dashed blue) and the approximation $(\omega_0+\varepsilon \nu, \|\varepsilon^{1/2}\varphi_0\|)$ (full red) starting at $\omega_0$.}}
	\label{DMD-PT_bifurcation_diagram}
\end{figure}

Figure \ref{DMD-PT_bifurcation_diagram} (a) shows the resulting potential $W(\cdot,\omega_0)$.  In Figure \ref{DMD-PT_bifurcation_diagram} (d) the bifurcation diagram is shown, where the actual (numerically computed) branch is plotted with the dashed blue line for $\omega \in (3.2,\omega_0)$, while the first order approximation $(\omega_0+\eps \nu, \|\varepsilon^{1/2}\varphi_0\|_{L^2})$ for $\eps\in (0,(\omega_0-3.2)/|\nu|)$ is plotted in full red. The numerical value of $\nu$ is $\nu \approx -7.4226$. A good agreement is observed between the asymptotic and the numerical curves in the vicinity of $\omega_0$. The eigenfunction $\varphi_0$ is plotted in Figure \ref{DMD-PT_bifurcation_diagram} (b). Finally, Figure \ref{DMD-PT_bifurcation_diagram} (c) shows the solution $\varphi$ at $\omega = 3.2,$ i.e. at the last $\omega$ in the continuation procedure.

\paragraph{Case 3} Here we choose $\hat\chi^{(1)}$ as in \eqref{chi1_w_3layers_homog}, \eqref{E:eta_MDM} with parameters as in Section \ref{S:3layers} ($k=2$, $\gamma^+=-\gamma^-=0.5$, and $\eta_*=0.2$). The linear eigenvalue is selected as $\omega_0=2.8096$. Again we have to check condition (Wt): a numerical approximation produces $\langle\partial_\omega W(\cdot,\omega_0)\varphi_0,\varphi_0^*\rangle\approx 6.202.$

The resulting bifurcation diagram is shown in Figure \ref{MDM_bifurcation_diagram} (d), where again the actual branch is plotted with the dashed blue line, while the first order approximation $(\omega_0+\eps \nu, \|\varepsilon^{1/2}\varphi_0\|)$ for $\eps\in (0,(\omega_0-1.5)/|\nu|)$ is plotted in full red. We get numerically $\nu \approx -2.7233$. Once again, a good agreement is observed between the asymptotic and the numerical curves. The eigenfunction $\varphi_0$ is plotted in Figure \ref{MDM_bifurcation_diagram} (b). Figure \ref{MDM_bifurcation_diagram} (c) shows the solution $\varphi$ at $\omega = 1.5,$ i.e. at the last $\omega$ in the continuation procedure.
\begin{figure}[h!]
	\centering
	\begin{subfigure}[b]{0.4\linewidth}
		\includegraphics[width=\linewidth]{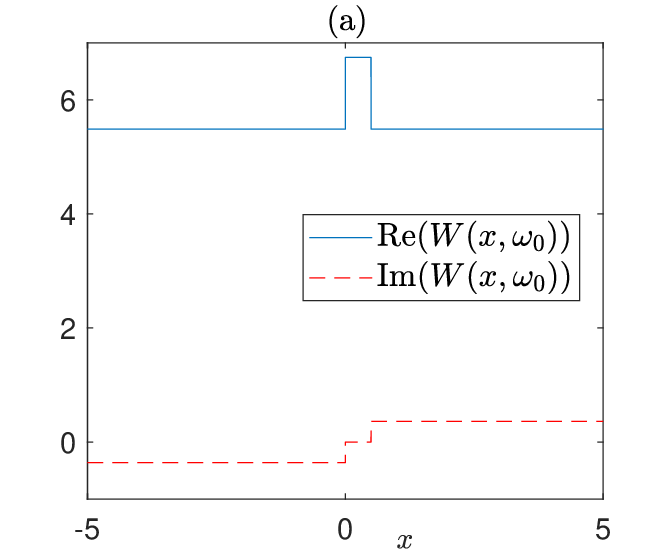}
	\end{subfigure}
	\quad
	\begin{subfigure}[b]{0.4\linewidth}
		\includegraphics[width=\linewidth]{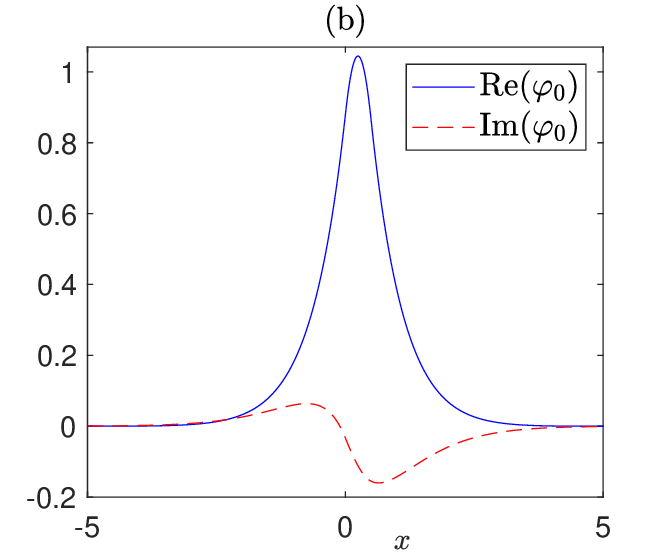}
	\end{subfigure}
	
	\begin{subfigure}[b]{0.4\linewidth}
		\includegraphics[width=\linewidth]{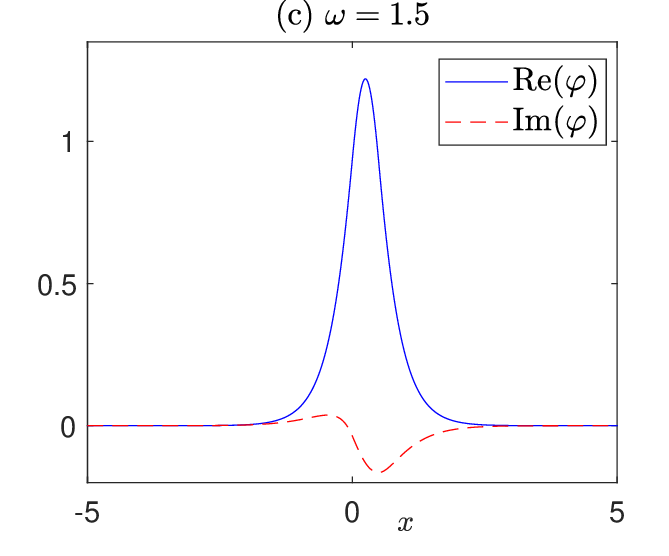}
	\end{subfigure}
	\quad
	\begin{subfigure}[b]{0.4\linewidth}
		\includegraphics[width=\linewidth]{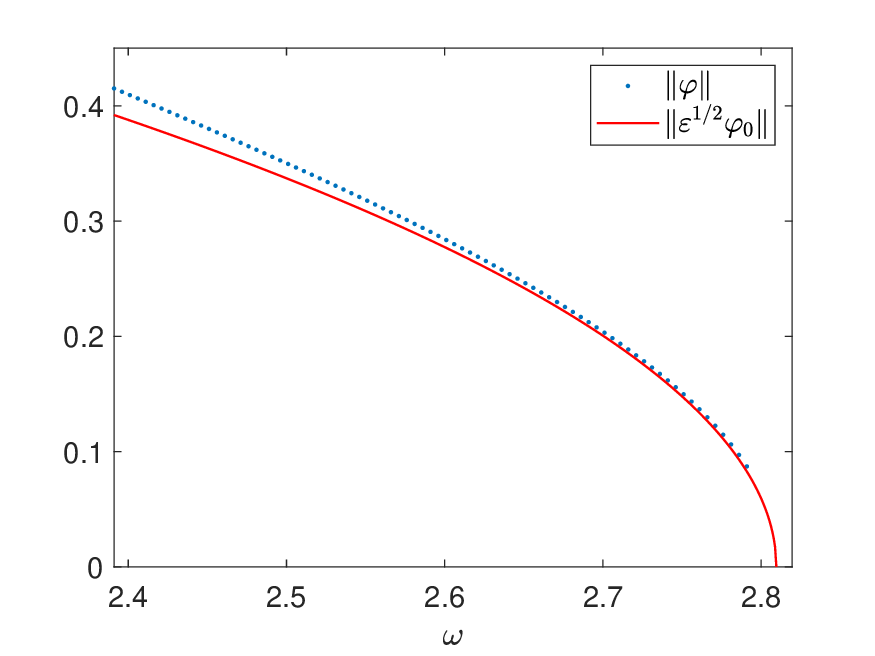}
	\end{subfigure}
	\caption{{\footnotesize Equation \eqref{NLpb}, \eqref{chi1_w_3layers_homog}, \eqref{E:eta_MDM} with the parameters: $d=0.5$, $\gamma^+=-\gamma^-=0.5$, and $\eta_*=0.2$. (a) The potential $W(\cdot,\omega_0)$ at $\omega_0\approx 2.8096$. (b) The eigenfunction $\varphi_0$ of \eqref{Lpb} at $\omega_0$. (c) The solution of \eqref{NLpb} for $\omega=1.5$. (d) The bifurcation diagram $(\omega, \|\varphi\|)$ (dashed blue) and the approximation $(\omega_0+\varepsilon \nu, \|\varepsilon^{1/2}\varphi_0\|)$ (full red) starting at $\omega_0$.}}
	\label{MDM_bifurcation_diagram}
\end{figure}

\section*{Acknowledgments}
This research is supported by the \emph{German Research Foundation}, DFG grant No. DO1467/4-1.

\bibliographystyle{plain}
\bibliography{bibliography3}

\end{document}